\documentclass[11pt]{amsart}
\input xy
\xyoption{all}
\usepackage{amsmath,amsthm, amscd, amssymb, amsfonts}
\usepackage{tikz-cd}
\usetikzlibrary{cd}
\usepackage{graphicx}

\newcommand{\eq}{\normalcolor{}}

\newcommand{\stab}{{\text{Stab}}}

\newcommand{\otb}{{\overline{\otimes}}}
\newcommand{\ott}{{\widetilde{\otimes}}}

\newcommand{\Mo}{{\mathcal M}}
\newcommand{\No}{{\mathcal N}}

\newcommand{\ot}{{\otimes}}

\newcommand{\ca}{{\mathcal C}}
\newcommand{\D}{{\mathcal D}}
\newcommand{\M}{{\mathcal M}}
\newcommand{\N}{{\mathcal N}}

\newcommand{\A}{{\mathcal A}}

\newcommand{\Do}{{\mathcal D}}

\newcommand{\rev}{\rm{rev}}

\newcommand{\balph}{\overline{\alpha}}
\newcommand{\walph}{\tilde{\alpha}}
\newcommand{\blambda}{\overline{\lambda}}

\newcommand{\bt}{{\boxtimes}}
\newcommand{\btc}{{\boxtimes_{\C}}}

\newcommand{\ku}{{\Bbbk}}

\newcommand{\uno}{ \mathbf{1}}
\newcommand{\C}{{\mathcal C}}

\newcommand{\id}{\mbox{\rm id\,}}
\newcommand{\Id}{\mbox{\rm Id\,}}

\newcommand{\Res}{\mbox{\rm Res\,}}
\newcommand{\Ind}{\mbox{\rm Ind\,}}

\newcommand{\vect}{\mbox{\rm vect\,}}

\newcommand{\Fun}{\operatorname{Fun}}

\newcommand\Hom{\operatorname{Hom}}

\newcommand{\End}{\operatorname{End}}

\theoremstyle{plain}

\numberwithin{equation}{section}

\newtheorem{teo}{Theorem}[section]

\newtheorem{lema}[teo]{Lemma}

\newtheorem{cor}[teo]{Corollary}

\newtheorem{prop}[teo]{Proposition}

\theoremstyle{definition}

\newtheorem{defi}[teo]{Definition}

  \newtheorem{exa}[teo]{Example}

\theoremstyle{remark}

\newtheorem{rmk}[teo]{Remark}

\def\pf{\begin{proof}}

\def\epf{\end{proof}}

\theoremstyle{remark}

\begin{document}

\title[Equivalence classes of exact module categories]
{Equivalence classes of exact module categories over graded tensor categories}
\author[   Mej\'ia Casta\~no and Mombelli  ]{ Adriana Mej\'ia Casta\~no and Mart\'in Mombelli
 }
 
 \address{Universidad del Norte,
 \newline
 \indent
 Departamento de Matem\'atica y Estad\'istica,
 \newline
 \indent
 Barranquilla, Colombia
 }
  \email{ sighana25@gmail.com, mejiala@uninorte.edu.co
 \newline
 \indent\emph{URL:} https://sites.google.com/site/adrianamejia1209}
\address{Facultad de Matem\'atica, Astronom\'\i a y F\'\i sica
\newline \indent
Universidad Nacional de C\'ordoba
\newline
\indent CIEM -- CONICET
\newline \indent Medina Allende s/n
\newline
\indent (5000) Ciudad Universitaria, C\'ordoba, Argentina}

\email{martin10090@gmail.com, mombelli@mate.uncor.edu
\newline \indent\emph{URL:}\/ http://www.famaf.unc.edu.ar/ $\sim$mombelli}
\subjclass[2010]{18D10; 16T05}

\begin{abstract}
We describe equivalence classes of indecomposable exact module categories over a finite graded tensor category. When applied to a pointed fusion category, our results coincide with the ones obtained in \cite{Na1}.
\end{abstract}

\date{\today}
\maketitle

%\tableofcontents

\section{Introduction}

Let  $\C$ be a finite tensor category. A $\C$-\textit{module category} consists of an abelian category $\M$ equipped  with an action functor $\C\times \M\to \M$, satisfying certain associativity and unit axioms.  The theory of representations of tensor categories has proven to be a powerful tool. In \cite{EO}, the authors introduce the notion of \textit{exact module category}, and as an interesting problem, the classification of indecomposable exact module categories over a fixed finite tensor category.

Let $G$ be a finite group, and $\D=\oplus_{g\in G} \C_g$ be a $G$-graded tensor category. This family of tensor categories has been studied in \cite{ENO}. In \cite{MM} and \cite{Ga2} the authors classify semisimple indecomposable modules over a semisimple $G$-graded tensor category $\D$ in terms of semisimple indecomposable modules over $\C_1$ and certain  cohomological data. This paper is devoted to extend and explain this classification, in the non-semisimple setting,  using a different approach, inspired on  results of \cite[Section 8]{MM}. Our classification, when applied to a pointed fusion category, recovers the results obtained in \cite{Na1}. Although very few results in this paper are new, we believe that the presentation of the results is our main contribution. We tried to be as   self-contained as possible.

\medbreak

The contents of the paper are the following. In Section  \ref{Section:represent} we give an account of all the necessary preliminaries on finite tensor categories and their representations. We recall the notion of internal Hom as an important tool in the study of module categories. In Section \ref{Section:gradedt} we recall the definition of graded tensor category, and some results concerning the restriction and induction of module categories. In Section \ref{Section:modograd} we start with the classification of indecomposable exact module categories over a fixed $G$-graded tensor category $\D=\oplus_{g\in G} \C_g$.

We aim to recover results from \cite[Section 8]{MM}. We show that indecomposable exact module categories over $\D=\oplus_{g\in G} \C_g$ are parametrized by collections $(H, \{A_g\}_{g\in H}, \beta)$ where for any $h,f,g\in H$
\begin{itemize}
\item $H\subseteq G$ is a subgroup;
\item $A_1=A$ is an algebra in $\C=\C_1$ such that $\C_A$ is an indecomposable exact $\C$-module;

\item $A_h$ is an invertible $A$-bimodule in $\C_h$, for any $h\in H$;

\item for any $h,f\in H$, $\beta_{f,h}:A_{fh} \to A_f\ot_A A_h $ is a bimodule isomorphism such that
$$ (\id_{A_f}\ot_A\beta_{g,h})\beta_{f,gh}=\walph_{A_f,A_g,A_h}(\beta_{f,g}\ot_A \id_{A_h}) \beta_{fg,h}. $$
\end{itemize}
Here $\walph$ is the associativity of the tensor category ${}_A\D_A$, of $A$-bimodules in $\D$. We also prove that two collections as above $(H, \{A_g\}_{g\in H}, \beta)$,  $(F, \{B_f\}_{f\in F}, \gamma)$ give equivalent module categories, if and only if, there exists $g\in G$, and  an invertible $(B,A)$-bimodule $C\in {}_B(\C_g)_{A}$ such that 
\begin{itemize}

\item $F=gHg^{-1}$;
\item there are isomorphisms $\tau_h: B_{ghg^{-1}}\ot_B C \to C \ot_A A_h,$ for all $h\in H$ such that
 \begin{align*}
     (\id\ot \beta_{h,l})\tau_{hl}&=\walph_{C,A_h,A_l}(\tau_h\ot\id) \walph^{-1}_{B_{ghg^{-1}},C,A_l} (\id\ot \tau_l)\walph_{B_{ghg^{-1}},B_{glg^{-1}},C}\\&(\gamma_{ghg^{-1},glg^{-1}}\ot\id).
 \end{align*}

\end{itemize}
The collection $(H, \{A_g\}_{g\in H}, \beta)$ encodes information to  describe how the category $\C_A$ has structure of $\C_H=\oplus_{g\in H} \C_g$-module. 
\smallbreak

To a collection $(H, \{A_g\}_{g\in H}, \beta)$ we associate an exact indecomposable $\D$-module category. To do this, we give to the category $\C_A$ a structure of $\C_H$-module category. Thus, the desired  $\D$-module is $\D\boxtimes_{\C_H} \C_A.$ This explicit description, allows us, in Corrollary \ref{fiber-funct}, to classify fiber functors over $\D$.

\subsection*{Acknowledgments} The work of M.M. was partially supported by CONICET (Argentina), FONCyT and Secyt (UNC). We also want to thank the referee for his/her constructive remarks that improve the presentation of the paper and for spotting several errors in previous versions of the paper.

\section{Preliminaries and Notation}

We shall work over  an algebraically closed field $\ku$ of characteristic 0. All vector spaces are assumed to be over $\ku$. All categories considered in this work will be abelian $\ku$-linear.

\subsection{Finite tensor categories}\label{Subsection:preliminaries:tensorcat} For basic notions on the theory of finite tensor categories we refer to \cite{EGNO}, \cite{EO}. Let $\C$ be a finite tensor category over $\ku$ with the associativity constraint given by $\alpha$, left and right unit isomorphisms given by $r$ and $l$. Let  $A\in \ca$ be an algebra. We shall denote by $\ca_A, {}_A\ca $ and $  {}_A\ca_A$ the categories of right $A$-modules, left $A$-modules and $A$-bimodules in $\ca$, respectively.  If $V\in \ca_A$ is a right $A$-module with action given by $\rho_V:V\ot A\to V$, and  $W\in{}_A\ca$ is a left $A$-module with action given by $\lambda_W:A\ot W\to W$; then we shall denote by $\pi_{V,W}: V \ot W\to V\ot_A W$  the coequalizer of the maps
$$\rho_V\ot\id_W,\,  (\id_V\ot \lambda_W)\alpha_{V,A,W}:(V \ot A)\ot W\longrightarrow V\ot W.$$
If $X, Y, Z\in {}_A\ca_A$ are $A$-bimodules, then  there exists a unique isomorphism $\balph_{X,Y,Z}:(X\ot Y)\ot _A Z\to X\ot (Y\ot_A Z)$ such that the diagram
\begin{equation}\label{assoc-A0} 
\xymatrix{
(X\ot Y)\ot Z \ar[d]_{\pi_{X\ot Y, Z} } \ar[r]^{\alpha_{X,Y,Z}} &  X\ot (Y\ot Z) \ar[d]^{ \id_X\ot \pi_{Y,Z}}\\
(X\ot Y)\ot_A Z \ar[r]^{\,\,\,\,\balph_{X,Y,Z}\,\,\,\,} & X\ot (Y\ot_A Z)
}
\end{equation}
is commutative.  Also, there are  isomorphisms $\walph_{X,Y,Z}:(X\ot_A Y)\ot _A Z\to X\ot_A (Y\ot_A Z)$ such that the diagram
\begin{equation}\label{assoc-A} 
\xymatrix{
(X\ot Y)\ot Z \ar[d]_{(\pi_{X,Y}\ot\id_Z)\pi_{X\ot Y, Z} } \ar[r]^{\alpha_{X,Y,Z}} &  X\ot (Y\ot Z) \ar[d]^{ \pi_{X,Y\ot_A Z}(\id_X\ot \pi_{Y,Z})}\\
(X\ot_A Y)\ot_A Z \ar[r]^{\,\,\,\,\walph_{X,Y,Z}\,\,\,\,} & X\ot_A (Y\ot_A Z)
}
\end{equation}
is commutative. It is well-known that ${}_A\ca_A$ is a monoidal category with product given by $\ot_A$ and associativity constraint given by $\walph$.

For any tensor category $\C$, we shall denote by $\C^{\rev}$ the tensor category obtained from $\C$ and opposite monoidal product: $X\ot^{\rev} Y=Y\ot X$, for any $X, Y\in\C$.
\begin{defi}\label{defi:bar-act} If $A\in \ca$ is an algebra and $C$ is a left $A$-module with structure map $\lambda_C:A\ot C\to C$, then we shall denote by $\blambda_C:A\ot_A C\to C$ the unique map such that $\blambda_C\, \pi_{A,C}=\lambda_C$. Analogously, If $C$ is a right $C$-module with structure map $\rho_C:C\ot A\to C$, we shall denote by $\overline{\rho}_C:C\ot_A A\to C$ the unique map such that $ \overline{\rho}_C \pi_{C,A}= \rho.$ In particular, the multiplication map of the algebra $m_A:A\ot A\to A$, makes $A$  a left (or right) $A$-module, thus one can consider the isomorphism $\overline{m}_A:A\ot_A A\to A$.
\end{defi}
The following technical result will be needed later.

\begin{lema}\label{tech-alg} Let $A\in \ca$ be an algebra with unit $u:\uno\to A$, and $(C,\lambda)$ a left $A$-module.  Then for any $X\in \ca_A$
\begin{equation}\label{unit-algA} 
\pi_{X,C} (\id_X\ot \blambda_C) \balph_{X,A,C} ((\id_X\ot u)\ot \id_C)(r^{-1}_X\ot \id_C)=\id_{X\ot_A C}.
\end{equation}
\end{lema} 
\pf We have that, the left hand side of \eqref{unit-algA} composed with $\pi_{X,C}$ equals
\begin{align*} &=\pi_{X,C} (\id_X\ot \blambda_C) \balph_{X,A,C} ((\id_X\ot u)\ot \id_C)(r^{-1}_X\ot \id_C)\pi_{X,C}\\
&=\pi_{X,C} (\id_X\ot \blambda_C) \balph_{X,A,C} ((\id_X\ot u)\ot \id_C)\pi_{X\ot \uno,C}(r^{-1}_X\ot \id_C)\\
&=\pi_{X,C} (\id_X\ot \blambda_C) \balph_{X,A,C} \pi_{X\ot A,C}((\id_X\ot u)\ot \id_C)(r^{-1}_X\ot \id_C)\\
&=\pi_{X,C} (\id_X\ot \blambda_C)(\id_X\ot \pi_{ A,C}) \alpha_{X,A,C} ((\id_X\ot u)\ot \id_C)(r^{-1}_X\ot \id_C)\\
&=\pi_{X,C} (\id_X\ot \lambda_C) (\id_X\ot (u\ot \id_C))\alpha_{X,\uno,C}(r^{-1}_X\ot \id_C)\\
&=\pi_{X,C} (\id_X\ot l_C) \alpha_{X,\uno,C}(r^{-1}_X\ot \id_C)= \pi_{X,C}.
\end{align*}
The fourth equality follows from \eqref{assoc-A0}. This implies the Lemma.
\epf

\section{Representations of tensor categories}\label{Section:represent}

 Let $\C$ be a finite tensor category over $\ku$. A (left) \emph{module} over 
$\ca$ is a  finite   $\ku$-linear  abelian category $\Mo$ together with a $\ku$-bilinear 
bifunctor $\otb: \ca \times \Mo \to \Mo$, exact in each variable,  endowed with 
 natural associativity
and unit isomorphisms 
$$m_{X,Y,M}: (X\otimes Y)\otb M \to X \otb
(Y\otb M), \ \ \ell_M: \uno \otb M\to M.$$ These isomorphisms are subject to the following conditions:
\begin{equation}\label{left-modulecat1} m_{X, Y, Z\otb M}\; m_{X\otimes Y, Z,
M}= (\id_{X}\otb m_{Y,Z, M})\;  m_{X, Y\otimes Z, M}(\alpha_{X,Y,Z}\otb\id_M),
\end{equation}
\begin{equation}\label{left-modulecat2} (\id_{X}\otb l_M)m_{X,{\bf
1} ,M}= r_X \otb \id_M,
\end{equation} for any $X, Y, Z\in\C, M\in\M.$ Here $\alpha$ is the associativity constraint of $\C$.
Sometimes we shall also say  that $\Mo$ is a $\ca$-\emph{module} or a $\ca$-\emph{module category}.
 In a similar way, one can define  right modules and bimodules. See
 \cite{Gr}.
\medbreak

Let $\Mo$ and $\Mo'$ be a pair of $\C$-modules. We say that a functor $F:\Mo\to\Mo'$ is a \emph{module functor} if it is equipped with natural isomorphisms
$$c_{X,M}: F(X\otb M)\to
X\otb F(M),\ \ X\in  \ca, M\in \Mo,$$  such that
for any $X, Y\in
\ca$, $M\in \Mo$:
\begin{align}\label{modfunctor1}
(\id_X \otb  c_{Y,M})c_{X,Y\otb M}F(m_{X,Y,M}) &=
m_{X,Y,F(M)}\, c_{X\otimes Y,M},
\\\label{modfunctor2}
\ell_{F(M)} \,c_{\uno ,M} &=F(\ell_{M}).
\end{align}

There is a composition
of module functors: if $\Mo''$ is another $\C$-module and
$(G,d): \Mo' \to \Mo''$ is another module functor then the
composition
\begin{equation}\label{modfunctor-comp}
(G\circ F, e): \Mo \to \Mo'', \qquad  e_{X,M} = d_{X,F(M)}\circ
G(c_{X,M}),
\end{equation} is
also a module functor.

\smallbreak  
We denote by $\Fun_{\ca}(\Mo, \Mo')$ the category whose
objects are module functors $(F, c)$ from $\Mo$ to $\Mo'$. A
morphism between  module functors $(F,c)$ and $(G,d)\in\Fun_{\ca}(\Mo,
\Mo')$ is a \textit{natural module transformation}, that is, a natural transformation $\alpha: F \to G$ such
that for any $X\in \ca$, $M\in \Mo$:
\begin{gather}
\label{modfunctor3} d_{X,M}\alpha_{X\otb M} =
(\id_{X}\otb \alpha_{M})c_{X,M}.
\end{gather}
 Two module functors $F, G$ are \emph{equivalent} if there exists a natural module isomorphism
$\alpha:F \to G$.

\smallbreak
Two $\C$-modules $\Mo$ and $\Mo'$ are {\em equivalent} if there exist module functors $F:\Mo\to
\Mo'$, $G:\Mo'\to \Mo$, and natural module isomorphisms
$\Id_{\Mo} \to F\circ G$, $\Id_{\Mo'} \to G\circ F$.

A module is
{\em indecomposable} if it is not equivalent to a direct sum of
two no trivial modules. Recall from \cite{EO}, that  a
module $\Mo$ is \emph{exact} if $\Mo$  for any
projective object
$P\in \ca$ the object $P\otb M$ is projective in $\Mo$, for all
$M\in\Mo$.  If $\M, \N$ are $\C$-bimodules, we denote by $\M\boxtimes_\C \N$ the balanced tensor product over $\C$. See \cite{DSS}.

\medbreak

The next result seems to be well-known.

\begin{lema}\label{action-zero} Let $\M$ be a $\C$-module. If $X\in \C, M\in \M$ are non-zero objects, then $X\otb M\neq 0.$
\end{lema}
\pf Let us assume $X\otb M=0$. Since $X$ is non-zero, $coev_X$ is non-zero. The map
$$ M \xrightarrow{ l_M} \uno\otb M \xrightarrow{ coev_X\otb \id_M} (^{*}X\ot X)\otb M 
\xrightarrow{ m_{^{*}X,X,M}}{} ^{*}X\otb (X\otb M)=0, $$
is the zero morphism.   The coevaluation $coev_X$ is a monomorphism. Since  $ \otb$ is bi-exact, then
$coev_X\otb \id_M$ is a monomorphism. Thus, the above composition is also a monomorphism. Hence $M= 0$.\epf

The following example will be used throughout the paper. If $A\in \C$ is an algebra, the category $\C_A$ of right $A$-modules in $\C$ is a left $\C$-module category. The action $\C\times \C_A\to \C_A$ is the tensor product $(X,M)\mapsto X\ot M$, where the action of $A$ on $X\ot M$ is on the second tensorand.

\subsection{The internal Hom} Let $\C$ be a finite tensor category and $\M$ be  a $\C$-module. For any pair of objects $N,M\in\M$, the \emph{internal Hom} is an object $\underline{\Hom}_\C(N,M)\in \C$ representing the functor $\Hom_\M(-\otb N,M):\C\to \vect_\ku$. That is, there are natural isomorphisms
\begin{equation}\label{Hom-interno}\Hom_\C(X,\underline\Hom_{\C}(N,M))\simeq\Hom_\M(X\otb N,M),\text{ for all }X\in\C.\end{equation}

\begin{prop}\label{homint}\cite[Thm. 3.17]{EO}  For each object $0\neq M\in \M$ the internal Hom $A=\underline\Hom_\C(M,M)$ is an algebra in $\C$. If $N$ is a subobject of $M$ then $\underline\Hom_\C(M,N)$ is a right ideal of $A$. Moreover $\underline\Hom_\C(M,-):\M\to\C_A$ is a $\C$-module functor.
If $\M$ is  indecomposable exact, the functor 
$$\underline\Hom_\C(M,-):\M\to \C_A$$ is an equivalence of $\C$-modules.\qed
\end{prop}

Using the above Proposition, when dealing with indecomposable exact module categories, we can restricts ourself only with those of the form $\C_A$, for some algebra $A\in \C$. The next result was given in \cite{EO}. %We shall include part of the proof that will be needed later.

\begin{prop}\label{natu-bimod} Let $A, B\in \C$ be algebras such that $\ca_A$ and $\ca_B$ are exact $\ca$-module categories. The functor 
$$\Psi:{}_A\C_B\to \Fun_\C(\C_A, \C_B), $$ defined by $\Psi(D)(X)=X\ot_A D$ is an equivalence of  categories.
\end{prop}
\pf For any $D\in {}_A\C_B$, the functor $\Psi(D)$ is indeed a $\ca$-module functor, with module structure given by
\begin{equation}\label{c-for-bimodule}
    c^D_{X,Y}: \Psi(D)(X\ot Y)\to X\ot \Psi(D)(Y),
\end{equation} 
$$c^D_{X,Y}= \balph_{X,Y,D}.$$
Recall from \eqref{assoc-A0} the definition of $\balph$. If $D,C\in {}_A\C_B $ and $\gamma:D\to C$ is an $(A,B)$-bimodule morphism, then $\Psi(\gamma)_X:\Psi(D)(X)\to \Psi(C)(X)$ defined as
$$\Psi(\gamma)_X=\id_X\ot \gamma, \  \ X\in \C_A,$$ is a module natural transformation between module functors $\Psi(D)$ and $\Psi(C)$.  Let us prove that $\Psi$ is an equivalence of categories. Define $$\Phi: \Fun_\C(\C_A, \C_B)\to {}_A\C_B,\ \  \Phi(F,c)=F(A).$$  
Then $F(A)$ has a structure of $(A,B)-$bimodule, with left action $\lambda:A\ot F(A)\to F(A)$, $\lambda=F(m_A)c^{-1}_{A,A}$, where $m_A:A\ot A\to A$ is the product of the algebra $A.$
It remains to prove that there are natural isomorphisms
$$\Phi\circ \Psi\sim \Id,\,  \Psi\circ \Phi\sim \Id.$$
The equivalence $\Phi\circ \Psi\sim \Id$ is clear. Now, take $(F,c)\in \Fun_\C(\C_A, \C_B)$. Let us show that there is a natural module isomorphism $F\to -\ot_A F(A)$, and this imply that $ \Psi\circ \Phi\sim \Id.$ 

For any $X\in \ca_A$, define $\tilde{\beta}_X:X\ot F(A)\to F(X)$ as $\tilde{\beta}_X= F(\rho_X)c^{-1}_{X,A}$. Here $\rho_X:X\ot A\to X$ is the right action  on $X$. The map $\tilde{\beta}_X$ factorizes as 
\begin{equation}\label{factoriz-beta}
\xymatrix{X\ot F(A)\ar[rr]^{ \pi_{X, F(A)}}
\ar[dr]_{\tilde{\beta}_X}&& X\ot_A F(A)
\ar[dl]^{\beta_X}\\ & F(X)&}.
\end{equation}
To prove that such $\beta_X$ exists, we have to prove that
\begin{align}\label{coeq-beta}  \tilde{\beta}_X (\rho_X\ot \id)=\tilde{\beta}_X (\id_X\ot \lambda) \alpha_{X,A,F(A)}.
\end{align}
The left hand side of \eqref{coeq-beta} is equal to
\begin{align*} &= F(\rho_X)c^{-1}_{X,A}  (\rho_X\ot \id)\\
&= F(\rho_X) F(\rho_X\ot\id_A) c^{-1}_{X\ot A, A}.
\end{align*}
The second equality follows from the naturality of $c$. The right hand side  of \eqref{coeq-beta} is equal to
\begin{align*} &=  F(\rho_X)c^{-1}_{X,A} (\id_X \ot F(m_A)c^{-1}_{A,A}) \alpha_{X,A,F(A)}\\
&= F(\rho_X) F(\id_X\ot m_A)c^{-1}_{X,A\ot A} (\id_X \ot c^{-1}_{A,A})\alpha_{X,A,F(A)}\\
&= F(\rho_X(\id_X\ot m_A)) F(\alpha_{X,A,A}) c^{-1}_{X\ot A, A}\\
&=  F(\rho_X) F(\rho_X\ot\id_A) c^{-1}_{X\ot A, A}.
\end{align*}
The first equality follows from the definition of $\lambda$, the second one follows from the naturality of $c$, the third one follows from \eqref{modfunctor1}, and the last equality follows from the associativity of $\rho_X.$ This finishes the proof of \eqref{coeq-beta}.

Let us prove that $\beta$ is a module natural transformation. For this we have to prove that for any $X\in \ca$, $M\in \ca_A$ we have
\begin{align}\label{beta-module-nat} c_{X,M} \beta_{X\ot M}= (\id_X\ot \beta_M)   \balph_{X,M,F(A)}.
\end{align}
Indeed
\begin{align*} c_{X,M} \beta_{X\ot M} \pi_{X\ot M,F(A)}&=  c_{X,M}  F(\rho_{X\ot M})c^{-1}_{X\ot M,A}\\
&=  c_{X,M}  F(\id_X\ot \rho_M)F(\alpha_{X,M,A})c^{-1}_{X\ot M,A}\\
&=(\id_X\ot F(\rho_M))c_{X, M\ot A} F(\alpha_{X,M,A})c^{-1}_{M,A}\\
&= (\id_X\ot F(\rho_M)) (\id_X\ot c^{-1}_{M,A} \alpha_{X,M,F(A)}.
\end{align*}
The first equality follows from the definition of $\beta$, the second equality follows from the definition of $\rho_{X\ot M}$,  the third one follows from the naturality of $c$, and the last one follows from \eqref{modfunctor1}. On the other hand 
\begin{align*} (\id_X\ot \beta_M)   \balph_{X,M,F(A)}\pi_{X\ot M,F(A)}&= (\id_X\ot \beta_M) (\id_X\ot \pi_{M,F(A)} \alpha_{X,M,F(A)}\\
&=(\id_X\ot F(\rho_M)) (\id_X\ot c^{-1}_{M,A} \alpha_{X,M,F(A)}.
\end{align*} 
The first equality follows from \eqref{assoc-A0}, and the second one follows from the definition of $\beta$. It remains to prove that $\beta_M$ is an isomorphism for any $M\in\ca_A$. Since $\rho_M$ is an epimorphism, and $F$ is exact,  $ \tilde{\beta}_M$ is an epimorphism. Thus $\beta_M$ is an epimorphism.  It is not difficult to see that if $M$ is a free $A$-module, that is $M=X\ot A$, where the right action of $A$ is given in the second tensorand, then $\beta_M$ is an isomorphism. For arbitrary $M\in \ca_A$ there exists an exact sequence of right $A$-modules
$$X\ot A \to Y\ot A\to M\to 0. $$

Since $\ca_A$ is an exact $\ca$-module category, then $\ot_A$ is biexact, see \cite[Ex. 3.19]{EO}. Then we have a diagramm
\begin{equation*}
\begin{tikzcd}
 \ar{d}{\beta_{X\ot A}} \ar{r} ( X\ot A)\ot_A F(A) &\ar{d}{\beta_{Y\ot A}} ( Y\ot A)\ot_A F(A)\ar{r}{}& M\ot_A F(A)\ar{d}{\beta_M}\ar{r} & 0\\
 \ar{r}  F(X\ot A) &\ar{r} F(Y\ot A) &\ar{r} F(M)\ar{r} & 0,
\end{tikzcd}
\end{equation*}
with both rows exact. Since $\beta_{X\ot A}, \beta_{Y\ot A}$ are isomorphisms, chasing about the diagram, it follows that $\beta_M$ is an isomorphism.\epf
\section{Graded tensor categories}\label{Section:gradedt}

 An important family of examples of
tensor categories comes from group extensions. Given a finite group $G$, a (faithful) \emph{$G$-grading on a  finite tensor category $\D$} is a decomposition
$\D=\oplus_{g\in G} \C_g$, where $\C_g$ are non-zero full abelian subcategories of $\D$ such that
$$\ot:\C_g\times \C_h\to \C_{gh}, \text{ for all }
g, h\in G.$$

In this case, we say that $\D$ is  a $G$-\emph{extension} of $\C:=\C_1$. These extensions were studied and classified in \cite{ENO} in terms of the Brauer-Picard group of the category $\C$ and certain cohomological data.

For any subgroup $H\subseteq G$, we define $\ca_H=\oplus_{g\in H} \C_g$. It is a tensor subcategory of $\D$.

\begin{exa}\label{pointed-fusion} Let $G$ be a finite group  and  let $\omega\in H^3(G,\ku^{\times})$ be a 3-cocycle. The category $\C(G,\omega)$ has as objects finite dimensional spaces $V$ equipped with a  $G$-grading of vector spaces, this means that $V=\oplus_{g\in G} V_g.$ The associativity constraint is defined by
$$a_{X,Y,Z}((x\ot y)\ot z)= \omega(g,h,f)\, x\ot (y\ot z),$$
for any $X, Y, Z\in \C(G,\omega)$, and any homogeneous elements $x\in X_g, y\in Y_h, z\in Z_f$, $g,f,h\in G$. The tensor category $\C(G,\omega)$ is an example of a $G$-extension of the category of finite dimensional vector spaces $\vect_\ku$. More precisely, $\C(G,\omega)= \oplus_{g\in G} \vect_g$, where $\vect_g$ denotes the category of finite dimensional vector spaces supported in the component $g$.
\end{exa}

We list some important properties of graded tensor categories.

\begin{prop}
\label{Res} Assume $\D=\oplus_{g\in G} \C_g$ is a $G$-graded extension of $\C$. Let $g,h\in G$. The following statements hold. 
\begin{itemize}
    \item[1.]

 The tensor product of $\D$  induces an equivalence of $\C$-bimodules \begin{equation}\label{MM} M_{g,h}:\C_g\btc\C_h\to\C_{gh}, \quad M_{g,h}(X\boxtimes Y)= X\ot Y.\end{equation}

\item[2.] For any subgroup $H\subseteq G$, the tensor product of $\Do$ induces an equivalence of $\C$-bimodules $\ca_{gH}\simeq \ca_g\btc \ca_H$.

\item[3.]  For any $\D$-module $\N$, the action functor induces an equivalence $\C_g\btc\N\simeq\N$ as $\C$-modules.
 
\end{itemize}
\end{prop}

\begin{proof} 1.  It follows, \textit{mutatis mutandis}, from the proof of  \cite[Theorem 6.1]{ENO}, in the non-semisimple case. 

2. It follows from \eqref{MM}.

3. It follows using  the argument of \cite[Lemma 10]{MM}.

\end{proof}

\subsection{Induction and restriction of module categories}

Let $\D$ be a tensor category, and let $\C$ be a tensor subcategory of $\D$. Assume $\N$ is a $\C$-module.
We shall denote by $\Ind_{\C}^\D(\N):=\D\bt_{\C}\N$ the \emph{induced $\D$-module}, where the $\D$-action is induced by the tensor product of $\D$ \cite[Proposition 2.13]{Ga}. Let $\M$ be a $\D$-module, we shall denote by $Res^\D_\C\M$ the \emph{restricted $\C$-module}.

The following results seem to be well-known. The proofs follow the same line as in \cite{MM},  we include them  for completeness' sake. Part of it is contained in \cite[Proposition 12]{MM}, \cite[Corollary 9]{MM}, see also \cite[Proposition 3.7]{Ga}.

\begin{prop}\label{ind-sub} Assume we have a decomposition $\D=\C\oplus \C'$ as abelian categories, such that $\C$ is a tensor subcategory. Let  $\M$ be an indecomposable exact $\D$-module such that it decomposes as $\M=\oplus^n_{i=1} \M_i$, where $\M_i$ are indecomposable exact $\C$-modules. Assume also that every time we choose non-zero objects $X\in \C'$, $N\in \M_1$, then $X\otb N\in \oplus^n_{i=2} \M_i$. Take $0\neq M\in \M_1$, and $A=\underline{\Hom}_\D(M,M)$. Then $A\in \C$, $\M_1\simeq \C_A$, and there is an equivalence of $\D$-modules
$$\M\simeq \Ind_{\C}^\D(\C_A).$$
\end{prop}
\begin{proof} Take arbitrary $V\in \C'$. Then by  \eqref{Hom-interno}
$$\Hom_\D(V, A)\simeq \Hom_\M(V\otb M, M).$$
Since $V\otb M \in \oplus^n_{i=2} \M_i$, then $\Hom_\D(V, A)=0$, and $A\in \C$. If $X\in \C$, there are isomorphisms
$$\Hom_\C(X, \underline{\Hom}_\C(M,M))\simeq \Hom_{\M_1}(X\otb M, M), $$
$$\Hom_\D(X, \underline{\Hom}_\D(M,M))\simeq \Hom_{\M}(X\otb M, M). $$
Since $\Hom_{\M_1}(X\otb M, M)=\Hom_{\M}(X\otb M, M)$, then 
$$\Hom_\C(X, \underline{\Hom}_\C(M,M))\simeq
\Hom_\D(X, A)\simeq \Hom_\C(X, A) .$$
Whence, $A=\underline{\Hom}_\C(M,M)$, and $\M_1\simeq \C_A$ as $\C$-modules.  Since $\Mo\simeq \Do_A$ as $\Do$-module categories, then the equivalence $\M\simeq \Ind_{\C}^\D(\C_A)$ follows.
\end{proof}

\begin{lema}\label{acti-equivalence} Let $\D$ be a $G$-graded extension of $\C$ and let $\Mo$ be an indecomposable exact $\Do$-module such that $\Res_\C^\D\M$  decomposes as $\oplus^n_{i=1} \M_i$, where $\M_i$ are indecomposable exact $\C$-modules.  Let $H:=\text{Stab}(\M_1)=\{f\in G: \C_f\btc\M_1\simeq\M_1\text{ as $\C$-module categories}\}$. Then by restriction $\Mo_1$ has a structure of $\ca_H$-module category and the action functor induces an equivalence $\Do \boxtimes_{\ca_H} \Mo_1 \simeq \Mo$ as $\Do$-modules.
\end{lema} 
\pf  
By Proposition \ref{Res}(3), for each $g\in G$, $\C_g\btc\Res_\C^\D\M\simeq\Res_\C^\D\M$, then $$\bigoplus_{i=1}^n\C_g\btc \M_i\simeq \bigoplus_{i=1}^n \M_i.$$
For any $g\in G$ and $i$ there exist an index $i_g\in \{1,...,n\}$ such that $\C_g\btc \M_i\simeq \M_{i_g}$ as $\C$-modules. This imply that  $G$ acts on the index set $\{1,...,n\}=X$. 

 If this action is not transitive then we can split the set $X$ into two disjoint sets $X=X_1\cup X_2$, where $X_1$ is the $G$-orbit of some $i\in X$, and $X_2\neq \emptyset$. Define $\M_1=\bigoplus_{j\in X_1}\M_j$ and $\M_2=\bigoplus_{j\in X_2}\M_j$. Hence $\M=\M_1\oplus\M_2$.  Therefore $\M_1$ is a $\D$-submodule category of $\M$, which is an absurd since $\M$ is exact indecomposable. 
The equivalence now follows from Proposition \ref{ind-sub}.
\epf

In the following Lemma we include some properties of the induced and restricted modules categories.

\begin{lema}\label{ind-rest} Assume $\D=\oplus_{g\in G} \C_g$ is a $G$-graded extension of $\C=\C_1$. Let $\N$ be a $\C$-module and $\M$ a $\D$-module. The following statements hold.
\begin{enumerate}

    \item[1.] $\M$ is an exact $\D$-module, if and only if,  $Res^\D_\C\M$ is an exact $\C$-module. 
    
    \item[2.] If $\N$ is exact (indecomposable) $\C$-module then $\Ind_{\C}^\D(\N)$ is exact (indecomposable) $\D$-module.
     \item[3.]  If $Res_\C^\D\M$ is an indecomposable $\C$-module then $\M$ is an indecomposable $\D$-module. 
\end{enumerate}

\end{lema}
\begin{proof} 1. Since $\C\subset\D$ is a tensor subcategory, then it follows from \cite[Corollary 2.5]{EG},  that  if $\Res_\C^\D\M$ is exact $\C$-module, then $\M$ is exact as a $\D$-module. Now, assume $\M$ is exact as a $\D$-module. Let be $P\in\C$ a projective object  and $X\in\Res_\C^\D\M$. Since $P\in\D$ is projective, then $P\otb X$ is projective in $Res^\D_\C\M$, hence $Res^\D_\C\M$ is exact.

2. To prove the exactness of  $\Ind_\C^\D\N$ we follow the argument of \cite[Prop. 2.10]{DN}. Since $\Ind_\C^\D\N= \oplus_{g\in G} \C_g \boxtimes_\C \N$, then $\Ind_\C^\D\N$ is an exact $\D$-module, if and only if, $\C_g \boxtimes_\C \N$ is an exact $\C$-module for any $g\in G$. It follows from \cite[Lemma 3.30]{EO} that $\C_g \boxtimes_\C \N$ is an exact $\End_\C(\C_g)$-module. Using \cite[Prop. 4.2]{ENO}, since $\C_g$ is an invertible $\C$-bimodule,    $\End_\C(\C_g)\simeq \C$. Hence $\C_g \boxtimes_\C \N$ is an exact $\C$-module.

Assume know that $\N$  is an indecomposable $\C$-module, and  we can decompose $\Ind_\C^\D\N=\M_1\oplus \M_2$ as $\D$-modules. By restriction $\N\simeq \uno\bt_\C\N= \N_1\oplus\N_2$ as $\C$-modules, then $\N_1=0$ or $\N_2=0$. Suppose $\N_2=0$, thus $\N\subset \M_1$. Take a non-zero object $0\neq M\in \M_2$. We can assume $M\in \C_g\bt_\C\N$, for some $g\in G$.  Take $0\neq Y\in\C_{g^{-1}}$, then, by Lemma \ref{action-zero} $0\neq Y\overline\ot M\in\N$ is a non-zero object. Since the restriction of the tensor product maps $\C_{g^{-1}}\times \C_g\to \C$, then $Y\overline\ot M\in \N_2$, which contradicts our assumption.

3. It follows straightforward.
\end{proof}

\section{Module categories over $G$-graded tensor categories}\label{Section:modograd}

Let $G$ be a finite group and $\D=\oplus_{g\in G} \C_g$ be a $G$-graded extension of $\C=\ca_1$. 
In \cite{MM}, \cite{Ga}, \cite{Ga2} the authors, independently, classify semisimple indecomposable modules over $\D$ in 
terms of exact $\ca$-modules and certain cohomological data. In this section we shall present the classification of exact indecomposable $\D$-module categories. We shall not assume that $\D$ is strict, since later, we want to apply the classification on non strict tensor categories. 

For any algebra $A\in \Do$, and any $g\in G$, the category $(\ca_g)_A$ is the category of right $A$-modules inside $\ca_g.$
\begin{lema}\label{action-equiv0}  Assume $\C_A$ is an indecomposable exact $\C$-module. For any $g\in G$, $\C_g\boxtimes_\ca \C_A$ is an indecomposable exact $\C$-module and the tensor product of $\D$ induces an equivalence of left $\ca$-modules $$(\C_g)_A\simeq \C_g\boxtimes_\ca \C_A,$$ where $(\C_g)_A$ is the category of $A$-modules inside $\C_g$. 
\end{lema}
\pf In the proof of Lemma \ref{ind-rest} (2), we proved that $\C_g\boxtimes_\ca \C_A$ is an indecomposable exact $\C$-module.
The proof of the last equivalence follows, {\it mutatis mutandis}, from the proof of \cite[Lemma 24]{MM} in the non-semisimple case, using Proposition \ref{ind-sub}.
\epf

Assume $A\in\C$ is an algebra such that $\C_A$ is a  $\D$-module with action given by $\otb:\D\times \ca_A\to \ca_A$. We also assume that the restriction of this action of $\C$ is the monoidal product. That is $\otb\mid_{\ca}:\ca\times \ca_A\to \ca_A$ equals the monoidal product $\ot.$ The restriction of  action functor $\otb:\D\times \C_A\to \C_A$ to the category $\C_g \times \C_A $ gives $\C$-module functors
$$\otb_g:\C_g \boxtimes_\ca \C_A \to \C_A, \text{ for any }g\in G.$$ Since $\Res^\Do_{\C} \C_A=\C_A$, then $\otb_1$ is the restriction of the tensor product of $\C$.
\begin{lema}\label{action-equiv} Under the above hypothesis, for $f,g,h\in G$, there exist natural $\ca$-module isomorphisms
\begin{equation*}%\label{associat-mod-gr} 
m_{f,g}: \otb_{fg} (M_{f,g}\boxtimes \Id_{\C_A})\to \otb_f (\Id_{\C_f}\boxtimes \otb_g), \end{equation*}
such that for any $X\in\C_f,Y\in\C_g,Z\in\C_h$
\begin{multline} \label{mdgt1} (m_{f,g})_{X,Y,Z\otb M} (m_{fg,h})_{X\ot Y,Z,  M} =\\
=(\id_X\otb (m_{g,h})_{Y,Z,  M}) (m_{f,gh})_{X,Y\ot Z,  M} (\alpha_{X,Y,Z}\otb \id_M).
\end{multline}
\end{lema}

\begin{proof}  The restriction of the associativity of $\C_A$ produces the natural module isomorphisms $m_{f,g}$%: \otb_{fg} (M_{f,g}\boxtimes \Id_{\C_A})\to \otb_f (\Id_{\C_f}\boxtimes \otb_g)$
. Equation \eqref{left-modulecat1} implies \eqref{mdgt1}.
\end{proof}

The following definition is inspired by \cite[Section 8]{MM}.
\begin{defi}\label{def:typea} Assume $A\in \C$ is an algebra such that $\C_A$ is an indecomposable exact $\C$-module. A \textit{type-A datum} for the module category $\C_A$ is a collection $(H, \{A_h\}_{h\in H}, \beta)$ where for any $f, g, h\in H$
\begin{itemize}
\item $H\subseteq G$ is a subgroup;
\item $A_1=A$;

\item $A_h$ is an invertible $A$-bimodule in $\C_h$, that is $A_h\in {}_A(\C_h)_A$;

\item  $\beta_{f,h}:A_{fh}\to A_f\ot_A A_h$ is an $A$-bimodule isomorphism  such that 
\begin{equation}\label{beta-assoc}
(\id_{A_f}\ot_A\beta_{g,h})\beta_{f,gh}=\walph_{A_f,A_g,A_h}(\beta_{f,g}\ot_A \id_{A_h}) \beta_{fg,h},\end{equation}
\begin{equation}\label{beta-unit}
\blambda_{A_h}\beta_{1,h}=\id_{A_h}, \quad \overline{\rho}_{A_h}\beta_{h,1}=\id_{A_h}.
\end{equation}
\end{itemize}
\end{defi}

Let $B\in \C$ be another algebra, such that $\C_B$ is an indecomposable exact $\C$-module. A type-A datum  $(F, \{B_f\}_{f\in F}, \gamma)$ for the module category $\C_B$, is \textit{equivalent} to the type-A datum $(H, \{A_h\}_{h\in H}, \beta)$ for $\C_A$ if 
\begin{itemize}
\item $F=H$;
\item there exists an invertible $(B,A)$-bimodule $C\in {}_B\C_{A}$ together with $(B,A)$-bimodule isomorphisms $$\tau_h:B_h\ot_B C \to C \ot_A A_h, \text{ for any } h\in H;$$
 such that for any $ l,h\in H$
\begin{equation}\label{typeA-equiv}\begin{split}
(\id\ot \beta_{h,l})\tau_{hl}&=\walph_{C,A_h,A_l}(\tau_h\ot\id) \walph^{-1}_{B_h,C,A_l} (\id\ot \tau_l)\walph_{B_h,B_l,C}(\gamma_{h,l}\ot\id),\\
\overline{\rho}_C\tau_1&= \overline{\lambda}_C.\end{split}
\end{equation} 
\end{itemize}

\begin{prop}\label{typeA-algebra-struc} Let $(H, \{A_h\}_{h\in H}, \beta)$ be a type-A datum. The object $\widehat{A}=\oplus_{h\in H} A_h\in \C_H$ has an algebra structure.
\end{prop}
\pf The unit $u:\uno\to \widehat{A}$ is the unit of the algebra $A$ composed with the canonical inclusion $A\hookrightarrow \widehat{A}$. The multiplication $\widehat{m}: \widehat{A}\ot \widehat{A}\to \widehat{A}$ is defined as follows. For any $f,h\in H$, 
$\widehat{m}\mid_{A_h\ot A_f}: A_h\ot A_f\to A_{hf}$ is the composition
$$A_h\ot A_f\xrightarrow{\pi_{A_h,A_f}} A_h\ot_A A_f \xrightarrow{\beta^{-1}_{h,f}}  A_{hf}. $$
Let us prove the associativity of this product. Let $h,f,l\in H$, then
\begin{align*} \widehat{m} (\widehat{m}\ot \id_{\widehat{A}}) \mid_{A_h\ot A_f\ot A_l} &=\beta^{-1}_{hf,l}\pi_{A_{hf},A_l}( \beta^{-1}_{h,f} \pi_{A_h,A_f}\ot \id_{A_l})\\
&=\beta^{-1}_{hf,l}( \beta^{-1}_{h,f} \ot_A \id_{A_l}) \pi_{A_{h}\ot_A A_f,A_l}(\pi_{A_h,A_f}\ot \id_{A_l}).
\end{align*}
On the other hand,
\begin{align*} \widehat{m} ( \id_{\widehat{A}}\ot \widehat{m}) \alpha_{\widehat{A},\widehat{A},\widehat{A}} \mid_{A_h\ot A_f\ot A_l}&= \beta^{-1}_{h,fl} \pi_{A_{h},A_{fl}}(\id_{A_h}\ot \beta^{-1}_{f,l} \pi_{A_{f},A_{l}}) \alpha_{A_h,A_f,A_l}\\
&= \beta^{-1}_{h,fl} (\id_{A_h}\ot \beta^{-1}_{f,l}) \walph_{A_h,A_f,A_l},
\end{align*}
in the  second equality we are using the definition of $ \walph$ given in \eqref{assoc-A}. It follows from \eqref{beta-assoc} that $\widehat{m} (\widehat{m}\ot \id_{\widehat{A}})= \widehat{m} ( \id_{\widehat{A}}\ot \widehat{m}) \alpha_{\widehat{A},\widehat{A},\widehat{A}}.$
\epf

Our task will be to classify indecomposable exact $\Do$-modules in terms of certain equivalence classes of type-A data. We shall not assume $\Do$ is strict, except in the proof of some particular theorems, since we want to apply the classification on non-strict tensor categories.

\medbreak

For any $g, h\in G$,  and any subset $H\subseteq G$ we denote 
$$h^g=g^{-1}hg, \quad H^g=g^{-1}Hg.$$

\begin{lema}\label{typeA-datum-to-modcat0} Let $g\in G$. Assume $(H^g, \{A_{h}\}_{h\in H^g}, \beta)$ is a type-A datum for the module category $\C_A$. Then, the category $(\ca_g)_A$ has a structure of $\ca_H$-module.
\end{lema}
\pf   For $M\in (\ca_g)_A$ define the action as
\begin{equation}\label{mod-action-cA}
X\ott M= (X\ot M)\ot_A A_{(f^{-1})^g},
\end{equation}
for any $f\in H$, $X\in \C_f$. The right action of $A$ on $X\ot M$ is given on the second tensorand. The associativity isomorphisms are given by  
\begin{equation}\label{associ-typeA-b} m_{X,Y,M}:((X\ot Y)\ot M)\ot_A A_{((fh)^{-1})^g}\to (X\ot ((Y\ot M)\ot_A A_{(f^{-1})^g})) \ot_A A_{(h^{-1})^g},
\end{equation}
\begin{align*} m_{X,Y,M}= &(\balph_{X,Y\ot M, A_{(f^{-1})^g}}\ot\id)\walph^{-1}_{X\ot (Y\ot M), A_{(f^{-1})^g}, A_{(h^{-1})^g}}(\id\ot \beta_{(f^{-1})^g,(h^{-1})^g})\\&(\alpha_{X,Y,M }\ot \id ),
\end{align*}
for any $h,f\in H$ and $X\in\C_f,Y\in\C_h$, $M\in (\ca_g)_A$.  Now, we shall prove that equation \eqref{left-modulecat1} is fulfilled. The proof is a bit long but quite straightforward. As a saving-space measure,  we shall denote $\widetilde h=(h^{-1})^g$ for any $h\in H$. Also, the tensor product will be denoted by yuxtaposition $X\ot Y=XY$.    Let be $h,f,l\in H$ and $X\in\C_f,Y\in\C_h$, $Z\in\C_l$, $M\in (\ca_g)_A.$ In this case, the left hand side of \eqref{left-modulecat1} is equal to
\begin{align*} ((&\id_X\ot   m_{Y,Z,M}\eq)\ot_A\id_{A_{(f^{-1})^g}})  m_{X,Y\ot Z,M}\eq((\alpha_{X,Y,Z}\ot\id_{M})\ot_A \id_{A_{((fhl)^{-1})^g}})\\
&=((\id_X\ot \balph_{Y,ZM,A_{\widetilde h}}\ot_A\id_{A_{\widetilde l}})\ot_A\id_{A_{\widetilde f}})((\id_X\ot \walph^{-1}_{Y(ZM),A_{\widetilde h},A_{\widetilde l}})\ot_A\id_{A_{\widetilde f}})\\
& ((\id_X\ot \id_{Y(ZM)}\ot_A\beta_{\widetilde h,\widetilde l})\ot_A\id_{A_{\widetilde f}})((\id_X\ot \alpha_{Y,Z,M}\ot_A\id_{A_{\widetilde h\widetilde l}})\ot_A\id_{A_{\widetilde f}})\eq\\
& (\overline\alpha_{X,(YZ)M,A_{\widetilde h\widetilde l}}\ot_A\id_{A_{\widetilde f}})
\walph^{-1}_{X((YZ)M),A_{\widetilde h\widetilde l},A_{\widetilde f}}(\id_{X((YZ)M)}\ot_A\beta_{\widetilde h\widetilde l,\widetilde f})\\
& (\alpha_{X,YZ,M}\ot_A\id_{A_{\widetilde{fhl}}})
((\alpha_{X,Y,Z}\ot\id_{M})\ot_A \id_{A_{\widetilde{fhl}}})\\
&=((\id_X\ot \balph_{Y,ZM,A_{\widetilde h}}\ot_A\id_{A_{\widetilde l}})\ot_A\id_{A_{\widetilde f}})((\id_X\ot \walph^{-1}_{Y(ZM),A_{\widetilde h},A_{\widetilde l}})\ot_A\id_{A_{\widetilde f}})\\
&((\id_X\ot \alpha_{Y,Z,M}\ot_A\id_{A_{\widetilde h}\ot_A A_{\widetilde l}})\ot_A\id_{A_{\widetilde f}})   ((\id_X\ot \id_{(YZ)M}\ot_A\beta_{\widetilde h,\widetilde l})\ot_A\id_{A_{\widetilde f}})\\
& \circ (\overline\alpha_{X,(YZ)M,A_{\widetilde h\widetilde l}}\ot_A\id_{A_{\widetilde f}})\eq
\walph^{-1}_{X((YZ)M),A_{\widetilde h\widetilde l},A_{\widetilde f}}(\id_{X((YZ)M)}\ot_A\beta_{\widetilde h\widetilde l,\widetilde f})\\
&(\alpha_{X,YZ,M}\ot_A\id_{A_{\widetilde h\widetilde l\widetilde f}})\circ  
((\alpha_{X,Y,Z}\ot\id_{M})\ot_A \id_{A_{((flh)^{-1})^g}}).
\end{align*}
Using that $\balph$ and $\walph$ are natural we get that this last expression equals to
\begin{align*}
&=((\id_X\ot \balph_{Y,ZM,A_{\widetilde h}}\ot_A\id_{A_{\widetilde l}})\ot_A\id_{A_{\widetilde f}})((\id_X\ot \walph^{-1}_{Y(ZM),A_{\widetilde h},A_{\widetilde l}})\ot_A\id_{A_{\widetilde f}})\\
&((\id_X\ot \alpha_{Y,Z,M}\ot_A\id_{A_{\widetilde h}\ot_A A_{\widetilde l}})\ot_A\id_{A_{\widetilde f}})\circ (\overline\alpha_{X,(YZ)M,A_{\widetilde h}\ot_A A_{\widetilde l}}\ot_A\id_{A_{\widetilde f}})\\
& ((\id_X\ot \id_{(YZ)M}\ot_A\beta_{\widetilde h,\widetilde l})\ot_A\id_{A_{\widetilde f}})
\walph^{-1}_{X((YZ)M),A_{\widetilde h\widetilde l},A_{\widetilde f}}\eq(\id_{X((YZ)M)}\ot_A\beta_{\widetilde h\widetilde l,\widetilde f})\\
&(\alpha_{X,YZ,M}\ot_A\id_{A_{\widetilde h\widetilde l\widetilde f}})\circ  
((\alpha_{X,Y,Z}\ot\id_{M})\ot_A \id_{A_{\widetilde f\widetilde l\widetilde h}})\\
&=((\id_X\ot \balph_{Y,ZM,A_{\widetilde h}}\ot_A\id_{A_{\widetilde l}})\ot_A\id_{A_{\widetilde f}})((\id_X\ot \walph^{-1}_{Y(ZM),A_{\widetilde h},A_{\widetilde l}})\ot_A\id_{A_{\widetilde f}})\\
&((\id_X\ot \alpha_{Y,Z,M}\ot_A\id_{A_{\widetilde h}\ot_A A_{\widetilde l}})\ot_A\id_{A_{\widetilde f}})\circ (\overline\alpha_{X,(YZ)M,A_{\widetilde h}\ot_A A_{\widetilde l}}\ot_A\id_{A_{\widetilde f}})\\
&\walph^{-1}_{X((YZ)M),A_{\widetilde h}\ot_A A_{\widetilde l},A_{\widetilde f}}((\id_X\ot \id_{(YZ)M}\ot_A\beta_{\widetilde h,\widetilde l})\ot_A\id_{A_{\widetilde f}})(\id_{X((YZ)M)}\ot_A\beta_{\widetilde h\widetilde l,\widetilde f})\\
& (\alpha_{X,YZ,M}\ot_A\id_{A_{\widetilde h\widetilde l\widetilde f}})\circ 
((\alpha_{X,Y,Z}\ot\id_{M})\ot_A \id_{A_{\widetilde f\widetilde l\widetilde h}})\eq\\
&=((\id_X\ot \balph_{Y,ZM,A_{\widetilde h}}\ot_A\id_{A_{\widetilde l}})\ot_A\id_{A_{\widetilde f}})((\id_X\ot \walph^{-1}_{Y(ZM),A_{\widetilde h},A_{\widetilde l}})\ot_A\id_{A_{\widetilde f}})\\
&((\id_X\ot \alpha_{Y,Z,M}\ot_A\id_{A_{\widetilde h}\ot_A A_{\widetilde l}})\ot_A\id_{A_{\widetilde f}})\circ (\overline\alpha_{X,(YZ)M,A_{\widetilde h}\ot_A A_{\widetilde l}}\ot_A\id_{A_{\widetilde f}})\\
&\walph^{-1}_{X((YZ)M),A_{\widetilde h}\ot_A A_{\widetilde l},A_{\widetilde f}}((\id_X\ot \id_{(YZ)M}\ot_A\beta_{\widetilde h,\widetilde l})\ot_A\id_{A_{\widetilde f}}) (\id_{X((YZ)M)}\ot_A\beta_{\widetilde h\widetilde l,\widetilde f})\\
& (\id_X\ot\alpha^{-1}_{Y,Z,M}\ot_A\id_{A_{\widetilde h\widetilde l\widetilde f}})  \eq
(\alpha_{X,Y,ZM}\ot_A\id_{A_{\widetilde h\widetilde l\widetilde f}})(\alpha_{XY,Z,M}\ot_A\id_{A_{\widetilde h\widetilde l\widetilde f}}).
\end{align*}
The last equation follows from  the pentagon axiom. The last expression above equals to 
\begin{align*}
&=((\id_X\ot \balph_{Y,ZM,A_{\widetilde h}}\ot_A\id_{A_{\widetilde l}})\ot_A\id_{A_{\widetilde f}})((\id_X\ot \walph^{-1}_{Y(ZM),A_{\widetilde h},A_{\widetilde l}})\ot_A\id_{A_{\widetilde f}})\\
&((\id_X\ot \alpha_{Y,Z,M}\ot_A\id_{A_{\widetilde h}\ot_A A_{\widetilde l}})\ot_A\id_{A_{\widetilde f}})\circ (\overline\alpha_{X,(YZ)M,A_{\widetilde h}\ot_A A_{\widetilde l}}\ot_A\id_{A_{\widetilde f}})\\
&\walph^{-1}_{X((YZ)M),A_{\widetilde h}\ot_A A_{\widetilde l},A_{\widetilde f}} ((\id_X\ot \id_{(YZ)M}\ot_A\beta_{\widetilde h,\widetilde l})\ot_A\id_{A_{\widetilde f}})\\
& (\id_X\ot\alpha^{-1}_{Y,Z,M}\ot_A\id_{A_{\widetilde h\widetilde l}\ot_A A_{\widetilde f}})\eq(\id_{X(Y(ZM))}\ot_A\beta_{\widetilde h\widetilde l,\widetilde f}) 
(\alpha_{X,Y,ZM}\ot_A \id_{A_{\widetilde h\widetilde l\widetilde f}})\\
&(\alpha_{XY,Z,M}\ot_A\id_{A_{\widetilde h\widetilde l\widetilde f}})\\
&=((\id_X\ot \balph_{Y,ZM,A_{\widetilde h}}\ot_A\id_{A_{\widetilde l}})\ot_A\id_{A_{\widetilde f}})((\id_X\ot \walph^{-1}_{Y(ZM),A_{\widetilde h},A_{\widetilde l}})\ot_A\id_{A_{\widetilde f}})\\
& ((\id_X\ot \alpha_{Y,Z,M}\ot_A\id_{A_{\widetilde h}\ot_A A_{\widetilde l}})\ot_A\id_{A_{\widetilde f}}) (\overline\alpha_{X,(YZ)M,A_{\widetilde h}\ot_A A_{\widetilde l}}\ot_A\id_{A_{\widetilde f}})\eq\\
&\walph^{-1}_{X((YZ)M),A_{\widetilde h}\ot_A A_{\widetilde l},A_{\widetilde f}}
(\id_X\ot\alpha^{-1}_{Y,Z,M}\ot_A\id_{A_{\widetilde h\widetilde l}\ot_A A_{\widetilde f}})\\&((\id_X\ot \id_{Y(ZM)}\ot_A\beta_{\widetilde h,\widetilde l})\ot_A\id_{A_{\widetilde f}})
(\id_{X(Y(ZM))}\ot_A\beta_{\widetilde h\widetilde l,\widetilde f}) 
(\alpha_{X,Y,ZM}\ot_A\id_{A_{\widetilde h\widetilde l\widetilde f}})\\
&(\alpha_{XY,Z,M}\ot_A\id_{A_{\widetilde h\widetilde l\widetilde f}})\end{align*}\begin{align*}
&=((\id_X\ot \balph_{Y,ZM,A_{\widetilde h}}\ot_A\id_{A_{\widetilde l}})\ot_A\id_{A_{\widetilde f}})((\id_X\ot \walph^{-1}_{Y(ZM),A_{\widetilde h},A_{\widetilde l}})\ot_A\id_{A_{\widetilde f}})\\
& (\overline\alpha_{X,Y(ZM),A_{\widetilde h}\ot_A A_{\widetilde l}}\ot_A\id_{A_{\widetilde f}}) ((\id_X\ot \alpha_{Y,Z,M}\ot_A\id_{A_{\widetilde h}\ot_A A_{\widetilde l}})\ot_A\id_{A_{\widetilde f}})\\
& \walph^{-1}_{X((YZ)M),A_{\widetilde h}\ot_A A_{\widetilde l},A_{\widetilde f}}
(\id_X\ot\alpha^{-1}_{Y,Z,M}\ot_A\id_{A_{\widetilde h\widetilde l}\ot_A A_{\widetilde f}})\eq\\&((\id_X\ot \id_{Y(ZM)}\ot_A\beta_{\widetilde h,\widetilde l})\ot_A\id_{A_{\widetilde f}})
(\id_{X(Y(ZM))}\ot_A\beta_{\widetilde h\widetilde l,\widetilde f}) 
(\alpha_{X,Y,ZM}\ot_A\id_{A_{\widetilde h\widetilde l\widetilde f}})\\
&(\alpha_{XY,Z,M}\ot_A\id_{A_{\widetilde h\widetilde l\widetilde f}})\\
&=((\id_X\ot \balph_{Y,ZM,A_{\widetilde h}}\ot_A\id_{A_{\widetilde l}})\ot_A\id_{A_{\widetilde f}})((\id_X\ot \walph^{-1}_{Y(ZM),A_{\widetilde h},A_{\widetilde l}})\ot_A\id_{A_{\widetilde f}})\\
& (\overline\alpha_{X,Y(ZM),A_{\widetilde h}\ot_A A_{\widetilde l}}\ot_A\id_{A_{\widetilde f}})
\walph^{-1}_{X(Y(ZM)),A_{\widetilde h}\ot_A A_{\widetilde l},A_{\widetilde f}}\\& ( \id_{X(Y(ZM))}\ot_A(\beta_{\widetilde h,\widetilde l}\ot_A\id_{A_{\widetilde f}}))
(\id_{X(Y(ZM))}\ot_A\beta_{\widetilde h\widetilde l,\widetilde f})   
(\alpha_{X,Y,ZM}\ot_A\id_{A_{\widetilde h\widetilde l\widetilde f}})\\
&(\alpha_{XY,Z,M}\ot_A\id_{A_{\widetilde h\widetilde l\widetilde f}}).
\end{align*}
Here we have used the naturality of $\alpha$ and $\walph$. Using \eqref{beta-assoc} we get that the last expression above equals
\begin{align*}
&=((\id_X\ot \balph_{Y,ZM,A_{\widetilde h}}\ot_A\id_{A_{\widetilde l}})\ot_A\id_{A_{\widetilde f}})((\id_X\ot \walph^{-1}_{Y(ZM),A_{\widetilde h},A_{\widetilde l}})\ot_A\id_{A_{\widetilde f}})\\
& (\overline\alpha_{X,Y(ZM),A_{\widetilde h}\ot_A A_{\widetilde l}}\ot_A\id_{A_{\widetilde f}})
 \walph^{-1}_{X(Y(ZM)),A_{\widetilde h}\ot_A A_{\widetilde l},A_{\widetilde f}}
(\id_{X(Y(ZM))}\ot_A\walph^{-1}_{A_{\widetilde h},A_{\widetilde l},A_{\widetilde f}})\\
&(\id_{X(Y(ZM))}\ot_A(\id_{A_{\widetilde h}}\ot_A\beta_{\widetilde l,\widetilde f}))(\id_{X(Y(ZM))}\ot_A\beta_{\widetilde h,\widetilde l\widetilde f})(\alpha_{X,Y,ZM}\ot_A\id_{A_{\widetilde h\widetilde l\widetilde f}})\\
&(\alpha_{XY,Z,M}\ot_A\id_{A_{\widetilde h\widetilde l\widetilde f}})\\
&=((\id_X\ot \balph_{Y,ZM,A_{\widetilde h}}\ot_A\id_{A_{\widetilde l}})\ot_A\id_{A_{\widetilde f}})((\id_X\ot \walph^{-1}_{Y(ZM),A_{\widetilde h},A_{\widetilde l}})\ot_A\id_{A_{\widetilde f}})\\
& (\overline\alpha_{X,Y(ZM),A_{\widetilde h}\ot_A A_{\widetilde l}}\ot_A\id_{A_{\widetilde f}})
(\walph_{X(Y(ZM)),A_{\widetilde h},A_{\widetilde l}}\ot_A\id_{A_{\widetilde f}})\walph^{-1}_{X(Y(ZM))\ot_A A_{\widetilde h},A_{\widetilde l},A_{\widetilde f}}\\
&\walph^{-1}_{X(Y(ZM)),A_{\widetilde h},A_{\widetilde l}\ot_A A_{\widetilde f}})(\id_{X(Y(ZM))}\ot_A(\id_{A_{\widetilde h}}\ot_A\beta_{\widetilde l,\widetilde f})) (\id_{X(Y(ZM))}\ot_A\beta_{\widetilde h,\widetilde l\widetilde f})\\
& (\alpha_{X,Y,ZM}\ot_A\id_{A_{\widetilde h\widetilde l\widetilde f}})\eq(\alpha_{XY,Z,M}\ot_A\id_{A_{\widetilde h\widetilde l\widetilde f}}). 
\end{align*}
The last equality follows from the pentagon axiom for $\walph$. This last expression equals to
\begin{align*}
&=((\id_X\ot \balph_{Y,ZM,A_{\widetilde h}}\ot_A\id_{A_{\widetilde l}})\ot_A\id_{A_{\widetilde f}})((\id_X\ot \walph^{-1}_{Y(ZM),A_{\widetilde h},A_{\widetilde l}})\ot_A\id_{A_{\widetilde f}})\\
& (\overline\alpha_{X,Y(ZM),A_{\widetilde h}\ot_A A_{\widetilde l}}\ot_A\id_{A_{\widetilde f}})
(\walph_{X(Y(ZM)),A_{\widetilde h},A_{\widetilde l}}\ot_A\id_{A_{\widetilde f}})\walph^{-1}_{X(Y(ZM))\ot_A A_{\widetilde h},A_{\widetilde l},A_{\widetilde f}}\\
&\walph^{-1}_{X(Y(ZM)),A_{\widetilde h},A_{\widetilde l}\ot_A A_{\widetilde f}}) (\id_{X(Y(ZM))}\ot_A(\id_{A_{\widetilde h}}\ot_A\beta_{\widetilde l,\widetilde f}))(\alpha_{X,Y,ZM}\ot_A\id_{A_{\widetilde h}\ot_A A_{\widetilde l\widetilde f}})\\&(\id_{(XY)(ZM))}\ot_A\beta_{\widetilde h,\widetilde l\widetilde f})(\alpha_{XY,Z,M}\ot_A\id_{A_{\widetilde h\widetilde l\widetilde f}})\\
&=((\id_X\ot \balph_{Y,ZM,A_{\widetilde h}}\ot_A\id_{A_{\widetilde l}})\ot_A\id_{A_{\widetilde f}})((\id_X\ot \walph^{-1}_{Y(ZM),A_{\widetilde h},A_{\widetilde l}})\ot_A\id_{A_{\widetilde f}})\\
& (\overline\alpha_{X,Y(ZM),A_{\widetilde h}\ot_A A_{\widetilde l}}\ot_A\id_{A_{\widetilde f}})
(\walph_{X(Y(ZM)),A_{\widetilde h},A_{\widetilde l}}\ot_A\id_{A_{\widetilde f}})\walph^{-1}_{X(Y(ZM))\ot_A A_{\widetilde h},A_{\widetilde l},A_{\widetilde f}}\\
&\walph^{-1}_{X(Y(ZM)),A_{\widetilde h},A_{\widetilde l}\ot_A A_{\widetilde f}})(\alpha_{X,Y,ZM}\ot_A\id_{A_{\widetilde h}\ot_A (A_{\widetilde l}\ot_A A_{\widetilde f}})\\
&(\id_{(XY)(ZM)}\ot_A(\id_{A_{\widetilde h}}\ot_A\beta_{\widetilde l,\widetilde f})) (\id_{(XY)(ZM))}\ot_A\beta_{\widetilde h,\widetilde l\widetilde f})(\alpha_{XY,Z,M}\ot_A\id_{A_{\widetilde h\widetilde l\widetilde f}})\end{align*}\begin{align*}
&=((\id_X\ot \balph_{Y,ZM,A_{\widetilde h}}\ot_A\id_{A_{\widetilde l}})\ot_A\id_{A_{\widetilde f}})((\id_X\ot \walph^{-1}_{Y(ZM),A_{\widetilde h},A_{\widetilde l}})\ot_A\id_{A_{\widetilde f}})\\
& (\overline\alpha_{X,Y(ZM),A_{\widetilde h}\ot_A A_{\widetilde l}}\ot_A\id_{A_{\widetilde f}})
(\walph_{X(Y(ZM)),A_{\widetilde h},A_{\widetilde l}}\ot_A\id_{A_{\widetilde f}})\walph^{-1}_{X(Y(ZM))\ot_A A_{\widetilde h},A_{\widetilde l},A_{\widetilde f}}\\
&\walph^{-1}_{X(Y(ZM)),A_{\widetilde h},A_{\widetilde l}\ot_A A_{\widetilde f}})(\alpha_{X,Y,ZM}\ot_A\id_{A_{\widetilde h}\ot_A (A_{\widetilde l}\ot_A A_{\widetilde f}})\\
& (\id_{(XY)(ZM)}\ot_A(\id\ot_A\beta_{\widetilde l,\widetilde f}))\walph_{(XY)(ZM),A_{\widetilde h},A_{\widetilde l\widetilde f}}\eq(\balph^{-1}_{XY,ZM,A_{\widetilde h}}\ot_A\id_{A_{\widetilde l\widetilde f}})m_{XY,Z,M}.
\end{align*}
The last equality follows from \eqref{associ-typeA-b}. This last expression equals to 
\begin{align*}
&=((\id_X\ot \balph_{Y,ZM,A_{\widetilde h}}\ot_A\id_{A_{\widetilde l}})\ot_A\id_{A_{\widetilde f}})((\id_X\ot \walph^{-1}_{Y(ZM),A_{\widetilde h},A_{\widetilde l}})\ot_A\id_{A_{\widetilde f}})\\
& (\overline\alpha_{X,Y(ZM),A_{\widetilde h}\ot_A A_{\widetilde l}}\ot_A\id_{A_{\widetilde f}})
(\walph_{X(Y(ZM)),A_{\widetilde h},A_{\widetilde l}}\ot_A\id_{A_{\widetilde f}})\walph^{-1}_{X(Y(ZM))\ot_A A_{\widetilde h},A_{\widetilde l},A_{\widetilde f}}\\
&\walph^{-1}_{X(Y(ZM)),A_{\widetilde h},A_{\widetilde l}\ot_A A_{\widetilde f}})(\alpha_{X,Y,ZM}\ot_A\id_{A_{\widetilde h}\ot_A (A_{\widetilde l}\ot_A A_{\widetilde f}})\walph_{(XY)(ZM),A_{\widetilde h},A_{\widetilde l}\ot_A A_{\widetilde f}}\\
& (\id_{(XY)(ZM)}\ot_A(\id_{A_{\widetilde h}}\ot_A\beta_{\widetilde l,\widetilde f}))(\balph^{-1}_{XY,ZM,A_{\widetilde h}}\ot_A\id_{A_{\widetilde l\widetilde f}})   m_{XY,Z,M}\\
&=((\id_X\ot \balph_{Y,ZM,A_{\widetilde h}}\ot_A\id_{A_{\widetilde l}})\ot_A\id_{A_{\widetilde f}})((\id_X\ot \walph^{-1}_{Y(ZM),A_{\widetilde h},A_{\widetilde l}})\ot_A\id_{A_{\widetilde f}})\\
& (\overline\alpha_{X,Y(ZM),A_{\widetilde h}\ot_A A_{\widetilde l}}\ot_A\id_{A_{\widetilde f}})
(\walph_{X(Y(ZM)),A_{\widetilde h},A_{\widetilde l}}\ot_A\id_{A_{\widetilde f}})\walph^{-1}_{X(Y(ZM))\ot_A A_{\widetilde h},A_{\widetilde l},A_{\widetilde f}}\\
& \walph^{-1}_{X(Y(ZM)),A_{\widetilde h},A_{\widetilde l}\ot_A A_{\widetilde f}}) (\alpha_{X,Y,ZM}\ot_A\id_{A_{\widetilde h}\ot_A (A_{\widetilde l}\ot_A A_{\widetilde f}})\walph_{(XY)(ZM),A_{\widetilde h},A_{\widetilde l}\ot_A A_{\widetilde f}}\\
&(\balph^{-1}_{XY,ZM,A_{\widetilde h}}\ot_A\id_{A_{\widetilde l}\ot_A A_{\widetilde f}})(\id_{XY}\ot(\id_{ZM\ot_A A_{\widetilde h}}\ot_A\beta_{\widetilde l,\widetilde f}))m_{XY,Z,M}\\
&=((\id_X\ot \balph_{Y,ZM,A_{\widetilde h}}\ot_A\id_{A_{\widetilde l }})\ot_A\id_{A_{\widetilde f}})((\id_X\ot \walph^{-1}_{Y(ZM),A_{\widetilde h},A_{\widetilde l}})\ot_A\id_{A_{\widetilde f}})\\
& (\overline\alpha_{X,Y(ZM),A_{\widetilde h}\ot_A A_{\widetilde l}}\ot_A\id_{A_{\widetilde f}})
(\walph_{X(Y(ZM)),A_{\widetilde h},A_{\widetilde l}}\ot_A\id_{A_{\widetilde f}})\walph^{-1}_{X(Y(ZM))\ot_A A_{\widetilde h},A_{\widetilde l},A_{\widetilde f}}\\
& ((\alpha_{X,Y,ZM}\ot_A\id_{A_{\widetilde h}})\ot_A\id_{ A_{\widetilde l}\ot_A A_{\widetilde f}})
(\balph^{-1}_{XY,ZM,A_{\widetilde h}}\ot_A\id_{A_{\widetilde l}\ot_A A_{\widetilde f}})\\
&(\id_{XY}\ot(\id_{ZM\ot_A A_{\widetilde h}}\ot_A\beta_{\widetilde l,\widetilde f}))m_{XY,Z,M}.
\end{align*}
Here, we have used the naturality of $\walph$. Using a pentagon axiom for $\balph$, the last expression is equal to
\begin{align*}
&=((\id_X\ot \balph_{Y,ZM,A_{\widetilde h}}\ot_A\id_{A_{\widetilde l}})\ot_A\id_{A_{\widetilde f}})((\id_X\ot \walph^{-1}_{Y(ZM),A_{\widetilde h},A_{\widetilde l}})\ot_A\id_{A_{\widetilde f}})\\
& (\overline\alpha_{X,Y(ZM),A_{\widetilde h}\ot_A A_{\widetilde l}}\ot_A\id_{A_{\widetilde f}})
(\walph_{X(Y(ZM)),A_{\widetilde h},A_{\widetilde l}}\ot_A\id_{A_{\widetilde f}})\walph^{-1}_{X(Y(ZM))\ot_A A_{\widetilde h},A_{\widetilde l},A_{\widetilde f}}\\
&(\balph^{-1}_{X,Y(ZM),A_{\widetilde h}}\ot_A\id_{A_{\widetilde l}\ot_A A_{\widetilde f}})((\id_X\ot\balph^{-1}_{Y,ZM,A_{\widetilde h}}\ot_A\id_{ A_{\widetilde l}\ot_A A_{\widetilde f}})\\
& (\alpha_{X,Y,(ZM)\ot_A A_{\widetilde h}}\ot_A\id_{A_{\widetilde l}\ot_A A_{\widetilde f}})(\id_{XY}\ot(\id_{ZM\ot_A A_{\widetilde h}}\ot_A\beta_{\widetilde l,\widetilde f}))   m_{XY,Z,M}\\
&=((\id_X\ot \balph_{Y,ZM,A_{\widetilde h}}\ot_A\id_{A_{\widetilde l}})\ot_A\id_{A_{\widetilde f}})((\id_X\ot \walph^{-1}_{Y(ZM),A_{\widetilde h},A_{\widetilde l}})\ot_A\id_{A_{\widetilde f}})\\
& (\overline\alpha_{X,Y(ZM),A_{\widetilde h}\ot_A A_{\widetilde l}}\ot_A\id_{A_{\widetilde f}})
(\walph_{X(Y(ZM)),A_{\widetilde h},A_{\widetilde l}}\ot_A\id_{A_{\widetilde f}})\walph^{-1}_{X(Y(ZM))\ot_A A_{\widetilde h},A_{\widetilde l},A_{\widetilde f}}\\
&(\balph^{-1}_{X,Y(ZM),A_{\widetilde h}}\ot_A\id_{A_{\widetilde l}\ot_A A_{\widetilde f}})((\id_X\ot\balph^{-1}_{Y,ZM,A_{\widetilde h}}\ot_A\id_{ A_{\widetilde l}\ot_A A_{\widetilde f}})\\
& (\id_{X(Y((ZM)\ot_A A_{\widetilde h}))}\ot_A\beta_{\widetilde l,\widetilde f}))(\alpha_{X,Y,(ZM)\ot_A A_{\widetilde h}}\ot_A\id_{A_{\widetilde l\widetilde f}})   m_{XY,Z,M}.
\end{align*}
Here, again, we have used the naturality of $\alpha$. Using the definition of the associativity  \eqref{associ-typeA-b}, the above expresssion equals to
\begin{align*}
&=((\id_X\ot \balph_{Y,ZM,A_{\widetilde h}}\ot_A\id_{A_{\widetilde l}})\ot_A\id_{A_{\widetilde f}})((\id_X\ot \walph^{-1}_{Y(ZM),A_{\widetilde h},A_{\widetilde l}})\ot_A\id_{A_{\widetilde f}})\\
& (\overline\alpha_{X,Y(ZM),A_{\widetilde h}\ot_A A_{\widetilde l}}\ot_A\id_{A_{\widetilde f}})
(\walph_{X(Y(ZM)),A_{\widetilde h},A_{\widetilde i}}\ot_A\id_{A_{\widetilde f}})\walph^{-1}_{X(Y(ZM))\ot_A A_{\widetilde h},A_{\widetilde l},A_{\widetilde f}}\\
&(\balph^{-1}_{X,Y(ZM),A_{\widetilde h}}\ot_A\id_{A_{\widetilde l}\ot_A A_{\widetilde f}}) (\id_X\ot\balph^{-1}_{Y,ZM,A_{\widetilde h}}\ot_A\id_{ A_{\widetilde l}\ot_A A_{\widetilde f}})\\
& \walph_{(X(Y(ZM))\ot_A A_{\widetilde h}),A_{\widetilde l},A_{\widetilde f}} \eq(\balph^{-1}_{X,Y((ZM)\ot_A A_{\widetilde h}),A_{\widetilde l}}\ot_A\id_{A_{\widetilde f}})m_{X,Y,(ZM)\ot_A A_{\widetilde h}} m_{XY,Z,M}\\
&=((\id_X\ot \balph_{Y,ZM,A_{\widetilde h}}\ot_A\id_{A_{\widetilde l}})\ot_A\id_{A_{\widetilde f}})((\id_X\ot \walph^{-1}_{Y(ZM),A_{\widetilde h},A_{\widetilde l}})\ot_A\id_{A_{\widetilde f}})\\
& (\overline\alpha_{X,Y(ZM),A_{\widetilde h}\ot_A A_{\widetilde l}}\ot_A\id_{A_{\widetilde f}})
(\walph_{X(Y(ZM)),A_{\widetilde h},A_{\widetilde l}}\ot_A\id_{A_{\widetilde f}}) \walph^{-1}_{(X(Y(ZM)))\ot_A A_{\widetilde h},A_{\widetilde l},A_{\widetilde f}}\\
& (\balph^{-1}_{X,Y(ZM),A_{\widetilde h}}\ot_A\id_{A_{\widetilde l}\ot_A A_{\widetilde f}})
\walph_{(X(Y(ZM))\ot_A A_{\widetilde h}),A_{\widetilde l},A_{\widetilde f}}\\
&((\id_X\ot\balph^{-1}_{Y,ZM,A_{\widetilde h}}\ot_A\id_{ A_{\widetilde l}})\ot_A \id_{A_{\widetilde f}})\eq(\balph^{-1}_{X,Y((ZM)\ot_A A_{\widetilde h}),A_{\widetilde l}}\ot_A\id_{A_{\widetilde f}})\\
&m_{X,Y,(ZM)\ot_A A_{\widetilde h}} m_{XY,Z,M}\\
&=((\id_X\ot \balph_{Y,ZM,A_{\widetilde h}}\ot_A\id_{A_{\widetilde l}})\ot_A\id_{A_{\widetilde f}})((\id_X\ot \walph^{-1}_{Y(ZM),A_{\widetilde h},A_{\widetilde l}})\ot_A\id_{A_{\widetilde f}})\\
& (\overline\alpha_{X,Y(ZM),A_{\widetilde h}\ot_A A_{\widetilde l}}\ot_A\id_{A_{\widetilde f}})
(\walph_{X(Y(ZM)),A_{\widetilde h},A_{\widetilde l}}\ot_A\id_{A_{\widetilde f}})
\\&((\balph^{-1}_{X,Y(ZM),A_{\widetilde h}}\ot_A\id_{A_{\widetilde l}})\ot_A \id_{A_{\widetilde f}}) ((\id_X\ot\balph^{-1}_{Y,ZM,A_{\widetilde h}}\ot_A\id_{ A_{\widetilde l}})\ot_A \id_{A_{\widetilde f}})\\
& (\balph^{-1}_{X,Y((ZM)\ot_A A_{\widetilde h}),A_{\widetilde l}}\ot_A\id_{A_{\widetilde f}})   m_{X,Y,(ZM)\ot_A A_{\widetilde h}} m_{XY,Z,M}.
\end{align*}
Here we have used again the naturality of $\walph$ and $\balph$. The above expression equals to
\begin{align*}
&=((\id_X\ot \balph_{Y,ZM,A_{\widetilde h}}\ot_A\id_{A_{\widetilde l}})\ot_A\id_{A_{\widetilde f}}) ((\id_X\ot \walph^{-1}_{Y(ZM),A_{\widetilde h},A_{\widetilde l}})\ot_A\id_{A_{\widetilde f}})\\
&  (\overline\alpha_{X,Y(ZM),A_{\widetilde h}\ot_A A_{\widetilde l}}\ot_A\id_{A_{\widetilde f}})
(\walph_{X(Y(ZM)),A_{\widetilde h},A_{\widetilde l}}\ot_A\id_{A_{\widetilde f}})\\
& ((\balph^{-1}_{X,Y(ZM),A_{\widetilde h}}\ot_A\id_{A_{\widetilde l}})\ot_A \id_{A_{\widetilde f}})(\balph^{-1}_{X,(Y(ZM))\ot_A A_{\widetilde h}),A_{\widetilde l}}\ot_A\id_{A_{\widetilde f}})\\
&((\id_X\ot\balph^{-1}_{Y,ZM,A_{\widetilde h}}\ot_A\id_{ A_{\widetilde l}})\ot_A \id_{A_{\widetilde f}})m_{X,Y,(ZM)\ot_A A_h} m_{XY,Z,M}\\
&= ((\id_X\ot \balph_{Y,ZM,A_{\widetilde h}}\ot_A\id_{A_{\widetilde l}})\ot_A\id_{A_{\widetilde f}})
((\id_X\ot\balph^{-1}_{Y,ZM,A_{\widetilde h}}\ot_A\id_{ A_{\widetilde li}})\ot_A \id_{A_{\widetilde f}})\\
&m_{X,Y,(ZM)\ot_A A_{\widetilde h}} m_{XY,Z,M}\\
&=m_{X,Y,(ZM)\ot_A A_{\widetilde h}} m_{XY,Z,M},
\end{align*}
the second equality follows from the pentagon axiom for $\balph$. This finishes the proof of \eqref{left-modulecat1}.
For any $M\in (\C_g)_A$ define the unit isomorphisms of this module category by $\widetilde{l}_M:\uno\widetilde\ot M:=(\uno\ot M)\ot_A A\to M$ as $\widetilde{l}_M=\overline{\rho}_M(l_M\ot_A\id_A)$. Here $l_M$ is the left unit isomorphism of the tensor category $\D$. The left hand side of \eqref{beta-unit} in this case is equal to
\begin{align*}
    (&\id_X\ot\widetilde{l}_M)m_{X,\uno,M}= (\id_X\ot\overline{\rho}_M\ot_A\id_{A_{(h^{-1})^g}})((\id_X(l_M\ot_A\id_A))\ot_A\id_{A_{(h^{-1})^g}}) \\
    &(\balph_{X,\uno\ot M, A_{1}}\ot\id_{A_{(h^{-1})^g}})\walph^{-1}_{X\ot (\uno\ot M), A_{1}, A_{(h^{-1})^g}}(\id\ot \beta_{1,(h^{-1})^g})\\
    &(\alpha_{X,\uno,M }\ot_A \id_{A_{(h^{-1})^g}} )\\
    &=(\id_X\ot l_M\ot_A\id_{A_{(h^{-1})^g}}) (\id_X\ot\overline{\rho}_{\uno\ot M}\ot_A\id_{A_{(h^{-1})^g}})  (\balph_{X,\uno\ot M, A_{1}}\ot\id_{A_{(h^{-1})^g}})\\
    &\walph^{-1}_{X\ot (\uno\ot M), A_{1}, A_{(h^{-1})^g}}(\id\ot \beta_{1,(h^{-1})^g})(\alpha_{X,\uno,M }\ot_A \id_{A_{(h^{-1})^g}} )\\
&=(\id_X\ot l_M\ot_A\id_{A_{(h^{-1})^g}}) (\overline{\rho}_{X(\uno\ot M)}\ot_A\id_{A_{(h^{-1})^g}}) \walph^{-1}_{X\ot (\uno\ot M), A_{1}, A_{(h^{-1})^g}}\\
&(\id\ot \beta_{1,(h^{-1})^g})(\alpha_{X,\uno,M }\ot_A \id_{A_{(h^{-1})^g}} )\\
&=(\id\ot l_M\ot_A\id_{A_{(h^{-1})^g}}) (\id\ot_A\overline{\rho}_{A_{(h^{-1})^g}})(\id\ot \beta_{1,(h^{-1})^g})\eq(\alpha_{X,\uno,M }\ot_A \id_{A_{(h^{-1})^g}} )\\
&=(\id_X\ot l_M\ot_A\id_{A_{(h^{-1})^g}})(\alpha_{X,\uno,M }\ot_A \id_{A_{(h^{-1})^g}} )\\
&=r_X\ot_A\id_{A_{(h^{-1})^g}}, 
\end{align*}
where the first equality is by definition, the second uses the naturality of $\overline{\rho}$, the third equality follows from the property of $\overline{\rho}$, the fourth by triangle axiom of $\walph$, the fifth follows by \eqref{beta-unit} and the sixth by triangle axiom of $\alpha$. This implies  \eqref{left-modulecat2} is satisfied and the category $(\ca_g)_A$ has a structure of $\ca_H$-module.
\epf
As a particular case of the above result, if $(H, \{A_h\}_{h\in H}, \beta)$ is a type-A datum for the module category $\C_A$, then $\ca_A$ has a  structure of $\ca_H$-module. Recall from Proposition \ref{typeA-algebra-struc} that out of a type-A datum we can construct an algebra $\widehat{A}=\oplus_{h\in H} A_h$ in $\C_H$. 

\begin{prop}\label{alg-mod} There exists an equivalence $\C_A\simeq (\C_H)_{\widehat{A}}$ of $\C_H$-modules.
\end{prop}
\pf We shall prove that there is an algebra isomorphism $\underline\Hom(A,A)\simeq \widehat{A}$, hence the proof follows using  Proposition \ref{homint}. Let $h\in H$, $X\in\C_h$, then
\begin{align*} \Hom_{\C_A}(X\ott A, A)&= \Hom_{\C_A}((X\ot A)\ot_A A_{h^{-1}}, A)\\ &\simeq \Hom_{\C_A}(X\ot A, A\ot_A A_h) \simeq \Hom_{\C}(X, A_h)\\
& \simeq \Hom_{\C_H}(X, \widehat{A}).
\end{align*}
The first equality is the definition of the action $\ott$ given in \eqref{mod-action-cA}, the first isomorphism follows from $A_h$ being isomorphic to the dual of $A_{h^{-1}}$. It can be verified that the algebra structure given to the internal Hom, presented in \cite[Thm. 3.17]{EO}, coincides with the algebra structure of $\widehat{A}$ described in Proposition \ref{typeA-algebra-struc}.
\epf

\begin{lema}\label{typeA-datum-to-modcat} Assume $(H, \{A_h\}_{h\in H}, \beta)$ and $(H, \{B_f\}_{f\in H}, \gamma)$
are type-A data for the modules $\C_A$ and $\ca_B$, respectively. Then the data $(H, \{A_h\}_{h\in H}, \beta)$ and $(H, \{B_f\}_{f\in H}, \gamma)$ are equivalent, if and only if, there exists a $\ca_H$-module equivalence $\ca_B\simeq \ca_A$.
\end{lema}
\pf If $C\in {}_B\C_{A}$ is an invertible bimodule, the functor $F_C:\ca_B\to \ca_A$, $F_C(M)=M\ot_B C$ is an equivalence of categories. Moreover, if $C$ comes from the equivalence of type-A data $(H, \{A_h\}_{h\in H}, \beta), (H, \{B_f\}_{f\in H}, \gamma)$, according to Definition \ref{def:typea}, then $F_C$ has a $\ca$-module functor structure given by
\begin{align*} c_{X,M}&:F_C(X\otb M)\to X\otb F_C(M)\\
c_{X,M}&=(\balph_{X,M,C}\ot \id)\walph^{-1}_{X\ot M,C,A_{f^{-1}}}(\id\ot \tau_{f^{-1}})\walph_{X\ot M, B_{f^{-1}},C}, \ \ f\in H, X\in \ca_f.
\end{align*}
Now, we shall prove that equation \eqref{modfunctor1} is fulfilled. Let be $X\in \C_f, Y\in\C_g$, $M\in(\C_g)_A$. The left hand side of \eqref{modfunctor1} in this case is equal to
\begin{align*}
    (&\id_X\ot c_{Y,M})c_{X,YM}F(m_{X,Y,M})=((\id_X\ot\balph_{Y,M,C}\ot_A\id_{A_g})\ot_A\id_{A_f})\\
    &(\id_X\ot\walph^{-1}_{YM,C,A_g}\ot_A\id_{A_f})((\id_X\ot\id_{YM}\ot_B\tau_g)\ot_A\id_{A_f})\\
    &(\id_X\ot \walph_{YM,A_g,C}\ot_A\id_{A_f})\circ (\balph_{X,(YM)\ot_B B_g,C}\ot_A\id_{A_f})\walph^{-1}_{X((YM)\ot_B B_g),C,A_f}\\
    &(\id_{X((YM)\ot_B B_g)}\ot_B\tau_f) \walph_{X((YM)\ot_B B_g),B_f,C}\circ (\balph_{X,YM,B_g}\ot_B\id_{B_f}\ot_B \id_C)\\
    &(\walph^{-1}_{X(YM),B_g,B_f}\ot_B \id_C)(\id_{X(YM)}\ot_B\gamma_{g,f}\ot_B \id_C)(\alpha_{X,Y,M}\ot_B\id_{B_{gf}}\ot_B\id_C)\\
    &=((\id_X\ot\balph_{Y,M,C}\ot_A\id_{A_g})\ot_A\id_{A_f})(\id_X\ot\walph^{-1}_{YM,C,A_g}\ot_A\id_{A_f})\\
    &((\id_X\ot\id_{YM}\ot_B\tau_g)\ot_A\id_{A_f})(\id_X\ot \walph_{YM,A_g,C}\ot_A\id_{A_f})\circ \\
    &(\balph_{X,(YM)\ot_B B_g,C}\ot_A\id_{A_f})\walph^{-1}_{X((YM)\ot_B B_g),C,A_f}(\id_{X((YM)\ot_B B_g)}\ot_B\tau_f)\\
    &(\balph_{X,YM,B_g}\ot_B\id_{B_f\ot_B C}) \walph_{(X(YM))\ot_B B_g,B_f,C}(\walph^{-1}_{X(YM),B_g,B_f}\ot_B \id_C)\eq\\
    &(\id_{X(YM)}\ot_B\gamma_{g,f}\ot_B \id_C)(\alpha_{X,Y,M}\ot_B\id_{B_{gf}}\ot_B\id_C).
\end{align*}
The last equation follows from the naturality of $\walph$. Using a pentagon axiom for $\walph$, the last expression is equal to
\begin{align*}
&=((\id_X\ot\balph_{Y,M,C}\ot_A\id_{A_g})\ot_A\id_{A_f})(\id_X\ot\walph^{-1}_{YM,C,A_g}\ot_A\id_{A_f})\\
    &((\id_X\ot\id_{YM}\ot_B\tau_g)\ot_A\id_{A_f})(\id_X\ot \walph_{YM,A_g,C}\ot_A\id_{A_f})\circ \\
    &(\balph_{X,(YM)\ot_B B_g,C}\ot_A\id_{A_f})\walph^{-1}_{X((YM)\ot_B B_g),C,A_f}(\id_{X((YM)\ot_B B_g)}\ot_B\tau_f)\\ &(\balph_{X,YM,B_g}\ot_B\id_{B_f\ot_B C})\walph^{-1}_{(X(YM)), B_g,B_f\ot_B C}(\id_{X(YM)}\ot_B\walph_{B_g,B_f,C})\\
    & \walph_{X(YM),B_g\ot_B B_F,C}(\id_{X(YM)}\ot_B\gamma_{g,f}\ot_B \id_C)\eq(\alpha_{X,Y,M}\ot_B\id_{B_{gf}}\ot_B\id_C)\\
&=((\id_X\ot\balph_{Y,M,C}\ot_A\id_{A_g})\ot_A\id_{A_f})(\id_X\ot\walph^{-1}_{YM,C,A_g}\ot_A\id_{A_f})\\
&((\id_X\ot\id_{YM}\ot_B\tau_g)\ot_A\id_{A_f})(\id_X\ot \walph_{YM,A_g,C}\ot_A\id_{A_f})\circ \\
&(\balph_{X,(YM)\ot_B B_g,C}\ot_A\id_{A_f})\walph^{-1}_{X((YM)\ot_B B_g),C,A_f}(\id_{X((YM)\ot_B B_g)}\ot_B\tau_f)\\ &(\balph_{X,YM,B_g}\ot_B\id_{B_f\ot_B C})\walph^{-1}_{(X(YM)), B_g,B_f\ot_B C}(\id_{X(YM)}\ot_B\walph_{B_g,B_f,C})\\
&(\id_{X(YM)}\ot_B\gamma_{g,f}\ot_B \id_C) \walph_{X(YM),B_{gf},C}(\alpha_{X,Y,M}\ot_B\id_{B_{gf}}\ot_B\id_C)\end{align*}\begin{align*}
&=((\id_X\ot\balph_{Y,M,C}\ot_A\id_{A_g})\ot_A\id_{A_f})(\id_X\ot\walph^{-1}_{YM,C,A_g}\ot_A\id_{A_f})\\
&((\id_X\ot\id_{YM}\ot_B\tau_g)\ot_A\id_{A_f})(\id_X\ot \walph_{YM,A_g,C}\ot_A\id_{A_f})\circ \\
&(\balph_{X,(YM)\ot_B B_g,C}\ot_A\id_{A_f})\walph^{-1}_{X((YM)\ot_B B_g),C,A_f}(\id_{X((YM)\ot_B B_g)}\ot_B\tau_f)\\ &(\balph_{X,YM,B_g}\ot_B\id_{B_f\ot_B C})\walph^{-1}_{(X(YM)), B_g,B_f\ot_B C}(\id_{X(YM)}\ot_B\walph_{B_g,B_f,C})\\
& (\id_{X(YM)}\ot_B\gamma_{g,f}\ot_B \id_C)(\alpha_{X,Y,M}\ot_B\id_{B_{gf}\ot_B C}) \eq\walph_{X(YM),B_{gf},C}.
\end{align*}
We have used the naturality of $\walph$. The above expression equals to
\begin{align*}
&=((\id_X\ot\balph_{Y,M,C}\ot_A\id_{A_g})\ot_A\id_{A_f})(\id_X\ot\walph^{-1}_{YM,C,A_g}\ot_A\id_{A_f})\\
&((\id_X\ot\id_{YM}\ot_B\tau_g)\ot_A\id_{A_f})(\id_X\ot \walph_{YM,A_g,C}\ot_A\id_{A_f})\circ \\
&(\balph_{X,(YM)\ot_B B_g,C}\ot_A\id_{A_f})\walph^{-1}_{X((YM)\ot_B B_g),C,A_f}(\id_{X((YM)\ot_B B_g)}\ot_B\tau_f)\\ &(\balph_{X,YM,B_g}\ot_B\id_{B_f\ot_B C})\walph^{-1}_{(X(YM)), B_g,B_f\ot_B C} (\id_{X(YM)}\ot_B\walph_{B_g,B_f,C})\\
& (\alpha_{X,Y,M}\ot_B\id_{B_{g}\ot_B B_f\ot_B C}) \eq(\id_{X(YM)}\ot_B\gamma_{g,f}\ot_B \id_C)\eq\walph_{X(YM),B_{gf},C}\\
&=((\id_X\ot\balph_{Y,M,C}\ot_A\id_{A_g})\ot_A\id_{A_f})(\id_X\ot\walph^{-1}_{YM,C,A_g}\ot_A\id_{A_f})\\
&((\id_X\ot\id_{YM}\ot_B\tau_g)\ot_A\id_{A_f})(\id_X\ot \walph_{YM,A_g,C}\ot_A\id_{A_f})\circ \\
&(\balph_{X,(YM)\ot_B B_g,C}\ot_A\id_{A_f})\walph^{-1}_{X((YM)\ot_B B_g),C,A_f} (\id_{X((YM)\ot_B B_g)}\ot_B\tau_f)\\ & (\balph_{X,YM,B_g}\ot_B\id_{B_f\ot_B C})\eq\walph^{-1}_{(X(YM)), B_g,B_f\ot_B C}(\alpha_{X,Y,M}\ot_B\id_{B_{g}\ot_B B_f\ot_B C})\\ 
&(\id_{(XY)M}\ot_B\walph_{B_g,B_f,C}) \eq(\id_{X(YM)}\ot_B\gamma_{g,f}\ot_B \id_C)\eq\walph_{X(YM),B_{gf},C}\\
&=((\id_X\ot\balph_{Y,M,C}\ot_A\id_{A_g})\ot_A\id_{A_f})(\id_X\ot\walph^{-1}_{YM,C,A_g}\ot_A\id_{A_f})\\
&((\id_X\ot\id_{YM}\ot_B\tau_g)\ot_A\id_{A_f})(\id_X\ot \walph_{YM,A_g,C}\ot_A\id_{A_f})\circ \\
&(\balph_{X,(YM)\ot_B B_g,C}\ot_A\id_{A_f}) \walph^{-1}_{X((YM)\ot_B B_g),C,A_f}(\balph_{X,YM,B_g}\ot_B\id_{C\ot_A A_f})\eq\\
&(\id_{(X(YM))\ot_B B_g}\ot_B\tau_f)\walph^{-1}_{(X(YM)), B_g,B_f\ot_B C}(\alpha_{X,Y,M}\ot_B\id_{B_{g}\ot_B B_f\ot_B C})\\ 
&(\id_{(XY)M}\ot_B\walph_{B_g,B_f,C}) \eq(\id_{X(YM)}\ot_B\gamma_{g,f}\ot_B \id_C)\eq\walph_{X(YM),B_{gf},C}\\
&=((\id_X\ot\balph_{Y,M,C}\ot_A\id_{A_g})\ot_A\id_{A_f})(\id_X\ot\walph^{-1}_{YM,C,A_g}\ot_A\id_{A_f})\\
&((\id_X\ot\id_{YM}\ot_B\tau_g)\ot_A\id_{A_f})(\id_X\ot \walph_{YM,A_g,C}\ot_A\id_{A_f})\circ \\
&(\balph_{X,(YM)\ot_B B_g,C}\ot_A\id_{A_f})(\balph_{X,YM,B_g}\ot_B\id_{C\ot_A A_f})\walph^{-1}_{(X(YM))\ot_B B_g),C,A_f}\\
& (\id_{(X(YM))\ot_B B_g}\ot_B\tau_f)\walph^{-1}_{(X(YM)), B_g,B_f\ot_B C}\eq(\alpha_{X,Y,M}\ot_B\id_{B_{g}\ot_B B_f\ot_B C})\\ 
&(\id_{(XY)M}\ot_B\walph_{B_g,B_f,C}) \eq(\id_{X(YM)}\ot_B\gamma_{g,f}\ot_B \id_C)\eq\walph_{X(YM),B_{gf},C}\end{align*}\begin{align*}
&=((\id_X\ot\balph_{Y,M,C}\ot_A\id_{A_g})\ot_A\id_{A_f})(\id_X\ot\walph^{-1}_{YM,C,A_g}\ot_A\id_{A_f})\\
&((\id_X\ot\id_{YM}\ot_B\tau_g)\ot_A\id_{A_f})(\id_X\ot \walph_{YM,A_g,C}\ot_A\id_{A_f})\circ \\
&(\balph_{X,(YM)\ot_B B_g,C}\ot_A\id_{A_f})(\balph_{X,YM,B_g}\ot_B\id_{C\ot_A A_f}) \walph^{-1}_{(X(YM))\ot_B B_g),C,A_f}\\
& \walph^{-1}_{X(YM),B_g,C\ot_A A_f}\eq(\id_{X(YM)}\ot_B \id_{B_g}\ot_B\tau_f)\eq(\alpha_{X,Y,M}\ot_B\id_{B_{g}\ot_B B_f\ot_B C})\\ 
&(\id_{(XY)M}\ot_B\walph_{B_g,B_f,C}) \eq(\id_{X(YM)}\ot_B\gamma_{g,f}\ot_B \id_C)\eq\walph_{X(YM),B_{gf},C}.
\end{align*}
The last equations follow from the naturality of $\walph$. Now, using a pentagon axiom for $\walph$, the last expression equals to
\begin{align*}
&=((\id_X\ot\balph_{Y,M,C}\ot_A\id_{A_g})\ot_A\id_{A_f})(\id_X\ot\walph^{-1}_{YM,C,A_g}\ot_A\id_{A_f})\\
&((\id_X\ot\id_{YM}\ot_B\tau_g)\ot_A\id_{A_f})(\id_X\ot \walph_{YM,A_g,C}\ot_A\id_{A_f})(\balph_{X,(YM)\ot_B B_g,C}\ot_A\id)\\
&(\balph_{X,YM,B_g}\ot_B\id_{C\ot_A A_f})(\walph^{-1}_{X(YM),B_g,C}\ot_A A_f)\walph^{-1}_{X(YM),B_g\ot_B C, A_f}\\
&(\id_{X(YM)}\ot_B\walph^{-1}_{B_g,C,A_f}) (\id_{X(YM)}\ot_B \id_{B_g}\ot_B\tau_f)(\alpha_{X,Y,M}\ot_B\id_{B_{g}\ot_B B_f\ot_B C})\\ 
&(\id_{(XY)M}\ot_B\walph_{B_g,B_f,C}) \eq(\id_{X(YM)}\ot_B\gamma_{g,f}\ot_B \id_C)\eq\walph_{X(YM),B_{gf},C}\\
&=((\id_X\ot\balph_{Y,M,C}\ot_A\id_{A_g})\ot_A\id_{A_f})(\id_X\ot\walph^{-1}_{YM,C,A_g}\ot_A\id_{A_f})\\
&((\id_X\ot\id_{YM}\ot_B\tau_g)\ot_A\id_{A_f})(\id_X\ot \walph_{YM,A_g,C}\ot_A\id_{A_f})(\balph_{X,(YM)\ot_B B_g,C}\ot_A\id)\\
&(\balph_{X,YM,B_g}\ot_B\id_{C\ot_A A_f})(\walph^{-1}_{X(YM),B_g,C}\ot_A A_f)\walph^{-1}_{X(YM),B_g\ot_B C, A_f}\\
& (\id_{X(YM)}\ot_B\walph^{-1}_{B_g,C,A_f})(\alpha_{X,Y,M}\ot_B\id_{B_{g}\ot_B C\ot_A A_f})\eq(\id_{X(YM)}\ot_B \id_{B_g}\ot_B\tau_f)\\ 
&(\id_{(XY)M}\ot_B\walph_{B_g,B_f,C}) \eq(\id_{X(YM)}\ot_B\gamma_{g,f}\ot_B \id_C)\eq\walph_{X(YM),B_{gf},C}\\
&=((\id_X\ot\balph_{Y,M,C}\ot_A\id_{A_g})\ot_A\id_{A_f})(\id_X\ot\walph^{-1}_{YM,C,A_g}\ot_A\id_{A_f})\\
&((\id_X\ot\id_{YM}\ot_B\tau_g)\ot_A\id_{A_f})(\id_X\ot \walph_{YM,A_g,C}\ot_A\id_{A_f})(\balph_{X,(YM)\ot_B B_g,C}\ot_A\id)\\
&(\balph_{X,YM,B_g}\ot_B\id_{C\ot_A A_f})(\walph^{-1}_{X(YM),B_g,C}\ot_A A_f)\walph^{-1}_{X(YM),B_g\ot_B C, A_f}\\
&(\alpha_{X,Y,M}\ot_B\id_{B_{g}\ot_B C\ot_A A_f}) (\id_{X(YM)}\ot_B\walph^{-1}_{B_g,C,A_f})(\id_{X(YM)}\ot_B \id_{B_g}\ot_B\tau_f)\\ 
& (\id_{(XY)M}\ot_B\walph_{B_g,B_f,C}) (\id_{X(YM)}\ot_B\gamma_{g,f}\ot_B \id_C)\eq\walph_{X(YM),B_{gf},C}.
\end{align*}
Using \eqref{typeA-equiv} we get that the last expression equals to
\begin{align*}
&=((\id_X\ot\balph_{Y,M,C}\ot_A\id_{A_g})\ot_A\id_{A_f})(\id_X\ot\walph^{-1}_{YM,C,A_g}\ot_A\id_{A_f})\\
&((\id_X\ot\id_{YM}\ot_B\tau_g)\ot_A\id_{A_f})(\id_X\ot \walph_{YM,A_g,C}\ot_A\id_{A_f})(\balph_{X,(YM)\ot_B B_g,C}\ot_A\id_{A_f})\\
&(\balph_{X,YM,B_g}\ot_B\id_{C\ot_A A_f})(\walph^{-1}_{X(YM),B_g,C}\ot_A A_f)\walph^{-1}_{X(YM),B_g\ot_B C, A_f}\\
& (\alpha_{X,Y,M}\ot_B\id_{B_{g}\ot_B C\ot_A A_f})(\id_{(XY)M}\ot_B\tau^{-1}_{g}\ot_A \id_{A_f})\eq(\id_{(XY)M}\ot_B \walph^{-1}_{C,A_g,A_f})\\ 
&(\id_{(XY)M}\ot_B\id_C\ot_A\beta_{g,f}) (\id_{(XY)M}\ot_B\tau_{gf})\eq\walph_{X(YM),B_{gf},C}\end{align*}\begin{align*}
&=((\id_X\ot\balph_{Y,M,C}\ot_A\id_{A_g})\ot_A\id_{A_f})(\id_X\ot\walph^{-1}_{YM,C,A_g}\ot_A\id_{A_f})\\
&((\id_X\ot\id_{YM}\ot_B\tau_g)\ot_A\id_{A_f})(\id_X\ot \walph_{YM,A_g,C}\ot_A\id_{A_f})(\balph_{X,(YM)\ot_B B_g,C}\ot_A\id_{A_f})\\
&(\balph_{X,YM,B_g}\ot_B\id_{C\ot_A A_f})(\walph^{-1}_{X(YM),B_g,C}\ot_A A_f)\walph^{-1}_{X(YM),B_g\ot_B C, A_f}\\
&(\id_{X(YM)}\ot_B\tau^{-1}_{g}\ot_A \id_{A_f}) (\alpha_{X,Y,M}\ot_B\id_{C\ot_A A_g\ot_A A_f})(\id_{(XY)M}\ot_B \walph^{-1}_{C,A_g,A_f})\\ 
&(\id_{(XY)M}\ot_B\id_C\ot_A\beta_{g,f}) (\id_{(XY)M}\ot_B\tau_{gf})\eq\walph_{X(YM),B_{gf},C}\\
&=((\id_X\ot\balph_{Y,M,C}\ot_A\id_{A_g})\ot_A\id_{A_f})(\id_X\ot\walph^{-1}_{YM,C,A_g}\ot_A\id_{A_f})\\
&((\id_X\ot\id_{YM}\ot_B\tau_g)\ot_A\id_{A_f})(\id_X\ot \walph_{YM,A_g,C}\ot_A\id_{A_f})(\balph_{X,(YM)\ot_B B_g,C}\ot_A\id_{A_f})\\
&(\balph_{X,YM,B_g}\ot_B\id_{C\ot_A A_f})(\walph^{-1}_{X(YM),B_g,C}\ot_A\id_{A_f})\walph^{-1}_{X(YM),B_g\ot_B C, A_f}\\
&(\id_{X(YM)}\ot_B\tau^{-1}_{g}\ot_A \id_{A_f})(\id_{X(YM)}\ot_B \walph^{-1}_{C,A_g,A_f})(\alpha_{X,Y,M}\ot_B\id_{C\ot_A A_g\ot_A A_f})\\ 
&(\id_{(XY)M}\ot_B\id_C\ot_A\beta_{g,f})  (\id_{(XY)M}\ot_B\tau_{gf})\walph_{X(YM),B_{gf},C}.
\end{align*}
Now, using the definition of $c_{XY,M}$, the above expression equals to
\begin{align*}
&=((\id_X\ot\balph_{Y,M,C}\ot_A\id_{A_g})\ot_A\id_{A_f})(\id_X\ot\walph^{-1}_{YM,C,A_g}\ot_A\id_{A_f})\\
&((\id_X\ot\id_{YM}\ot_B\tau_g)\ot_A\id_{A_f})(\id_X\ot \walph_{YM,A_g,C}\ot_A\id_{A_f})(\balph_{X,(YM)\ot_B B_g,C}\ot_A\id_{A_f})\\
&(\balph_{X,YM,B_g}\ot_B\id_{C\ot_A A_f})(\walph^{-1}_{X(YM),B_g,C}\ot_A\id_{A_f})\walph^{-1}_{X(YM),B_g\ot_B C, A_f}\\
&(\id_{X(YM)}\ot_B\tau^{-1}_{g}\ot_A \id_{A_f})(\id_{X(YM)}\ot_B \walph^{-1}_{C,A_g,A_f})(\alpha_{X,Y,M}\ot_B\id_{C\ot_A A_g\ot_A A_f})\\ 
& (\id_{(XY)M}\ot_B\id_C\ot_A\beta_{g,f})  \walph_{(XY)M,C,A_{gf}}\eq(\balph^{-1}_{XY,M,C}\ot_A\id_{A_{fg}})c_{XY,M}\\
&=((\id_X\ot\balph_{Y,M,C}\ot_A\id_{A_g})\ot_A\id_{A_f})(\id_X\ot\walph^{-1}_{YM,C,A_g}\ot_A\id_{A_f})\\
&((\id_X\ot\id_{YM}\ot_B\tau_g)\ot_A\id_{A_f})(\id_X\ot \walph_{YM,A_g,C}\ot_A\id_{A_f})(\balph_{X,(YM)\ot_B B_g,C}\ot_A\id_{A_f})\\
&(\balph_{X,YM,B_g}\ot_B\id_{C\ot_A A_f})(\walph^{-1}_{X(YM),B_g,C}\ot_A\id_{A_f})\walph^{-1}_{X(YM),B_g\ot_B C, A_f}\\
&(\id_{X(YM)}\ot_B\tau^{-1}_{g}\ot_A \id_{A_f})(\id_{X(YM)}\ot_B \walph^{-1}_{C,A_g,A_f}) (\alpha_{X,Y,M}\ot_B\id_{C\ot_A A_g\ot_A A_f})\\ 
& \walph_{(XY)M,C,A_{gf}}\eq(\id_{((XY)M)\ot_B C}\ot_A\beta_{g,f}) \eq(\balph^{-1}_{XY,M,C}\ot_A\id_{A_{fg}})c_{XY,M}\\
&=((\id_X\ot\balph_{Y,M,C}\ot_A\id_{A_g})\ot_A\id_{A_f})(\id_X\ot\walph^{-1}_{YM,C,A_g}\ot_A\id_{A_f})\\
&((\id_X\ot\id_{YM}\ot_B\tau_g)\ot_A\id_{A_f})(\id_X\ot \walph_{YM,A_g,C}\ot_A\id_{A_f})(\balph_{X,(YM)\ot_B B_g,C}\ot_A\id_{A_f})\\
&(\balph_{X,YM,B_g}\ot_B\id_{C\ot_A A_f})(\walph^{-1}_{X(YM),B_g,C}\ot_A\id_{A_f})\walph^{-1}_{X(YM),B_g\ot_B C, A_f}\\
&(\id_{X(YM)}\ot_B\tau^{-1}_{g}\ot_A \id_{A_f}) (\id_{X(YM)}\ot_B \walph^{-1}_{C,A_g,A_f}) \walph_{X(YM),C,A_{g}\ot_A A_f}\\
&(\alpha_{X,Y,M}\ot_B\id_{C}\ot_A\id_{A_g\ot_A A_f})\eq(\id_{((XY)M)\ot_B C}\ot_A\beta_{g,f}) \eq(\balph^{-1}_{XY,M,C}\ot_A\id_{A_{fg}})c_{XY,M}.
\end{align*}
The last equations follow from naturality of $\walph$. Again, using a pentagon for $\walph$ the last expression equals to

\begin{align*}
&=((\id_X\ot\balph_{Y,M,C}\ot_A\id_{A_g})\ot_A\id_{A_f})(\id_X\ot\walph^{-1}_{YM,C,A_g}\ot_A\id_{A_f})\\
&((\id_X\ot\id_{YM}\ot_B\tau_g)\ot_A\id_{A_f})(\id_X\ot \walph_{YM,A_g,C}\ot_A\id_{A_f})(\balph_{X,(YM)\ot_B B_g,C}\ot_A\id_{A_f})\\
&(\balph_{X,YM,B_g}\ot_B\id_{C\ot_A A_f})(\walph^{-1}_{X(YM),B_g,C}\ot_A\id_{A_f})\walph^{-1}_{X(YM),B_g\ot_B C, A_f}\\
&(\id_{X(YM)}\ot_B\tau^{-1}_{g}\ot_A \id_{A_f})\walph_{X(YM),C,\ot_A A_g,A_f}(\walph_{X(YM),C,A_{g}}\ot_A \id_{A_f})\\
&\walph^{-1}_{(X(YM))\ot_B C,A_g,A_f}(\alpha_{X,Y,M}\ot_B\id_{C}\ot_A\id_{A_g\ot_A A_f}) (\id_{((XY)M)\ot_B C}\ot_A\beta_{g,f}) \\
& (\balph^{-1}_{XY,M,C}\ot_A\id_{A_{fg}})   c_{XY,M}\\
&=((\id_X\ot\balph_{Y,M,C}\ot_A\id_{A_g})\ot_A\id_{A_f})(\id_X\ot\walph^{-1}_{YM,C,A_g}\ot_A\id_{A_f})\\
&((\id_X\ot\id_{YM}\ot_B\tau_g)\ot_A\id_{A_f})(\id_X\ot \walph_{YM,A_g,C}\ot_A\id_{A_f})(\balph_{X,(YM)\ot_B B_g,C}\ot_A\id_{A_f})\\
&(\balph_{X,YM,B_g}\ot_B\id_{C\ot_A A_f})(\walph^{-1}_{X(YM),B_g,C}\ot_A\id_{A_f})\walph^{-1}_{X(YM),B_g\ot_B C, A_f}\\
&(\id_{X(YM)}\ot_B\tau^{-1}_{g}\ot_A \id_{A_f})\walph_{X(YM),C,\ot_A A_g,A_f}(\walph_{X(YM),C,A_{g}}\ot_A \id_{A_f})\\
&\walph^{-1}_{(X(YM))\ot_B C,A_g,A_f} (\alpha_{X,Y,M}\ot_B\id_{C}\ot_A\id_{A_g\ot_A A_f}) (\balph^{-1}_{((XY),M,C}\ot_A\id_{A_g\ot_A A_f}) \\
& (\alpha^{-1}_{X,Y,M\ot_B C}\ot_A\id_{A_g\ot_A A_f})\eq(\id_{(X(Y(M\ot_B C)))}\ot_A\beta_{g,f})(\alpha_{X,Y,M\ot_B C}\ot_A\id_{A_{gf}})   c_{XY,M}.
\end{align*}

Using a pentagon axiom for $\balph$ the above equation equals to
\begin{align*}
&=((\id_X\ot\balph_{Y,M,C}\ot_A\id_{A_g})\ot_A\id_{A_f})(\id_X\ot\walph^{-1}_{YM,C,A_g}\ot_A\id_{A_f})\\
&((\id_X\ot\id_{YM}\ot_B\tau_g)\ot_A\id_{A_f})(\id_X\ot \walph_{YM,A_g,C}\ot_A\id_{A_f})(\balph_{X,(YM)\ot_B B_g,C}\ot_A\id_{A_f})\\
&(\balph_{X,YM,B_g}\ot_B\id_{C\ot_A A_f})(\walph^{-1}_{X(YM),B_g,C}\ot_A\id_{A_f})\walph^{-1}_{X(YM),B_g\ot_B C, A_f}\\
&(\id_{X(YM)}\ot_B\tau^{-1}_{g}\ot_A \id_{A_f})\walph_{X(YM),C,\ot_A A_g,A_f}(\walph_{X(YM),C,A_{g}}\ot_A \id_{A_f})\\
&\walph^{-1}_{(X(YM))\ot_B C,A_g,A_f}(\alpha_{X,Y,M}\ot_B\id_{C}\ot_A\id_{A_g\ot_A A_f}) (\balph^{-1}_{X,(YM),C}\ot_A \id_{A_g\ot_A A_f})\\
& (\id_X\ot\balph^{-1}_{Y,M,C}\ot_A\id_{A_g\ot_A A_f})\eq(\id_{(X(Y(M\ot_B C)))}\ot_A\beta_{g,f})(\alpha_{X,Y,M\ot_B C}\ot_A\id_{A_{gf}})   c_{XY,M}\\
&=((\id_X\ot\balph_{Y,M,C}\ot_A\id_{A_g})\ot_A\id_{A_f})(\id_X\ot\walph^{-1}_{YM,C,A_g}\ot_A\id_{A_f})\\
&((\id_X\ot\id_{YM}\ot_B\tau_g)\ot_A\id_{A_f})(\id_X\ot \walph_{YM,A_g,C}\ot_A\id_{A_f})(\balph_{X,(YM)\ot_B B_g,C}\ot_A\id_{A_f})\\
&(\balph_{X,YM,B_g}\ot_B\id_{C\ot_A A_f})(\walph^{-1}_{X(YM),B_g,C}\ot_A\id_{A_f})\walph^{-1}_{X(YM),B_g\ot_B C, A_f}\\
&(\id_{X(YM)}\ot_B\tau^{-1}_{g}\ot_A \id_{A_f})\walph_{X(YM),C,\ot_A A_g,A_f}(\walph_{X(YM),C,A_{g}}\ot_A \id_{A_f})\\
&(\balph^{-1}_{X,(YM),C}\ot_A \id_{A_g\ot_A A_f})(\id_X\ot\balph^{-1}_{Y,M,C}\ot_A\id_{A_g\ot_A A_f}) \walph^{-1}_{X(Y(M\ot_B C)),A_g,A_f} \\
& (\id_{(X(Y(M\ot_B C)))}\ot_A\beta_{g,f})(\alpha_{X,Y,M\ot_B C}\ot_A\id_{A_{gf}})   c_{XY,M}.
\end{align*}
The above equal follows from the naturality of  $\walph$. Using the definition of $m_{X,Y,F(M)}$, the last expression above equals to

\begin{align*}
&=((\id_X\ot\balph_{Y,M,C}\ot_A\id_{A_g})\ot_A\id_{A_f})(\id_X\ot\walph^{-1}_{YM,C,A_g}\ot_A\id_{A_f})\\
&((\id_X\ot\id_{YM}\ot_B\tau_g)\ot_A\id_{A_f})(\id_X\ot \walph_{YM,A_g,C}\ot_A\id_{A_f})(\balph_{X,(YM)\ot_B B_g,C}\ot_A\id_{A_f})\\
&(\balph_{X,YM,B_g}\ot_B\id_{C\ot_A A_f})(\walph^{-1}_{X(YM),B_g,C}\ot_A\id_{A_f})\walph^{-1}_{X(YM),B_g\ot_B C, A_f}\\
&(\id_{X(YM)}\ot_B\tau^{-1}_{g}\ot_A \id_{A_f})\walph_{X(YM),C,\ot_A A_g,A_f}(\walph_{X(YM),C,A_{g}}\ot_A \id_{A_f})\\
&(\balph^{-1}_{X,(YM),C}\ot_A \id_{A_g\ot_A A_f}) (\id_X\ot\balph^{-1}_{Y,M,C}\ot_A\id_{A_g\ot_A A_f}) (\balph^{-1}_{X,(Y(M\ot_B C),A_g}\ot_A\id_{A_f}) \\
& m_{X,Y,(F(M)} c_{XY,M}\\
&=((\id_X\ot\balph_{Y,M,C}\ot_A\id_{A_g})\ot_A\id_{A_f})(\id_X\ot\walph^{-1}_{YM,C,A_g}\ot_A\id_{A_f})\\
&((\id_X\ot\id_{YM}\ot_B\tau_g)\ot_A\id_{A_f}) (\id_X\ot \walph_{YM,A_g,C}\ot_A\id_{A_f})(\balph_{X,(YM)\ot_B B_g,C}\ot_A\id_{A_f})\\
& (\balph_{X,YM,B_g}\ot_B\id_{C\ot_A A_f})(\walph^{-1}_{X(YM),B_g,C}\ot_A\id_{A_f})\eq\walph^{-1}_{X(YM),B_g\ot_B C, A_f}\\
&(\id_{X(YM)}\ot_B\tau^{-1}_{g}\ot_A \id_{A_f})\walph_{X(YM),C,\ot_A A_g,A_f}(\walph_{X(YM),C,A_{g}}\ot_A \id_{A_f})\\
&(\balph^{-1}_{X,(YM),C}\ot_A \id_{A_g\ot_A A_f})(\balph^{-1}_{X,(YM)\ot_B C,A_g}\ot_A\id_{A_f})\\
&((\id_X\ot\balph^{-1}_{Y,M,C}\ot_A\id_{A_g})\ot_A \id_{A_f}) m_{X,Y,(F(M)} c_{XY,M}.
\end{align*}
Again, by the naturality of $\balph$, we obtain the last equal. Using a pentagon axiom for $\balph$, the above expression equals to
\begin{align*}
&=((\id_X\ot\balph_{Y,M,C}\ot_A\id_{A_g})\ot_A\id_{A_f})(\id_X\ot\walph^{-1}_{YM,C,A_g}\ot_A\id_{A_f})\\
& ((\id_X\ot\id_{YM}\ot_B\tau_g)\ot_A\id_{A_f})(\balph_{X,YM,B_g\ot_B C}\ot_A \id_{A_g})\eq\walph^{-1}_{X(YM),B_g\ot_B C, A_f}\\
&(\id_{X(YM)}\ot_B\tau^{-1}_{g}\ot_A \id_{A_f})\walph_{X(YM),C,\ot_A A_g,A_f}(\walph_{X(YM),C,A_{g}}\ot_A \id_{A_f})\\
&(\balph^{-1}_{X,(YM),C}\ot_A \id_{A_g\ot_A A_f})(\balph^{-1}_{X,(YM)\ot_B C,A_g}\ot_A\id_{A_f})\\
&((\id_X\ot\balph^{-1}_{Y,M,C}\ot_A\id_{A_g})\ot_A \id_{A_f}) m_{X,Y,(F(M)} c_{XY,M}\\
&=((\id_X\ot\balph_{Y,M,C}\ot_A\id_{A_g})\ot_A\id_{A_f})(\id_X\ot\walph^{-1}_{YM,C,A_g}\ot_A\id_{A_f})\\
&(\balph_{X,YM,C\ot_A A_g}\ot_A \id_{A_f}) (\id_{X(YM)}\ot_B\tau_g\ot_A\id_{A_f})\walph^{-1}_{X(YM),B_g\ot_B C, A_f}\\
& (\id_{X(YM)}\ot_B\tau^{-1}_{g}\ot_A \id_{A_f})\walph_{X(YM),C,\ot_A A_g,A_f}\eq(\walph_{X(YM),C,A_{g}}\ot_A \id_{A_f})\\
&(\balph^{-1}_{X,(YM),C}\ot_A \id_{A_g\ot_A A_f})(\balph^{-1}_{X,(YM)\ot_B C,A_g}\ot_A\id_{A_f})\\
&((\id_X\ot\balph^{-1}_{Y,M,C}\ot_A\id_{A_g})\ot_A \id_{A_f}) m_{X,Y,(F(M)} c_{XY,M}\\
&=((\id_X\ot\balph_{Y,M,C}\ot_A\id_{A_g})\ot_A\id_{A_f}) (\id_X\ot\walph^{-1}_{YM,C,A_g}\ot_A\id_{A_f})\\
& (\balph_{X,YM,C\ot_A A_g}\ot_A \id_{A_f})(\walph_{X(YM),C,A_{g}}\ot_A \id_{A_f})\\
& (\balph^{-1}_{X,(YM),C}\ot_A \id_{A_g\ot_A A_f})(\balph^{-1}_{X,(YM)\ot_B C,A_g}\ot_A\id_{A_f})\\
&((\id_X\ot\balph^{-1}_{Y,M,C}\ot_A\id_{A_g})\ot_A \id_{A_f})   m_{X,Y,(F(M)} c_{XY,M}.
\end{align*}
The last equals follow from naturality of  $\balph$ and $\walph$. Finally, using a pentagon axiom for $\balph$, the above expression is the right hand of \eqref{modfunctor1}.
In the same way, we obtain \eqref{modfunctor2} and $F$ is a $\C$-module functor. Then equivalent data give rise to equivalent module categories.

\medbreak

 Now, suppose that $(F,c):\C_B\to\C_A$ is a $\C_H$-module equivalence. Since $F$ is a module functor, it follows from \cite[Proposition 3.11]{EO}, that $F$ is exact, hence there exist an invertible object $C\in{}_B\C_A$  such that $F=-\ot_B C$. The module structure on $F$ produces natural isomorphisms
 $$\overline{c}_{X,M}:((X\ot_B M)\ot_B B_{h^{-1}})\ot_B C\to (X\ot_B (M\ot_B C))\ot_A A_{h^{-1}},$$
 for any $h\in H$, $X\in \C_h$, $M\in {}_B\C_B$. To see this fact, one has to verify that the isomorphisms 
 $$c_{X,M}:((X\ot M)\ot_B B_{h^{-1}})\ot_B C\to (X\ot (M\ot_B C))\ot_A A_{h^{-1}}, $$
 composed with the map $$\pi_{X,M\ot_BC}\ot \id_{A_{h^{-1}}}:(X\ot (M\ot_B C))\ot_A A_{h^{-1}}\to (X\ot_B (M\ot_B C))\ot_A A_{h^{-1}},$$
 factorizes through the map
 $$\pi_{X,M}\ot \id_{B_{h^{-1}}}\ot\id_C:((X\ot M)\ot_B B_{h^{-1}})\ot_B C\to ((X\ot_B M)\ot_B B_{h^{-1}})\ot_B C. $$

  For any $h\in H$, the isomorphisms $\overline{c}_{X,M}$ induce $(B,A)$-bimodule isomorphisms $\tau_h:B_h\ot_B C\to C\ot_A A_h$, given by
  \begin{align*}\tau_h=&(\overline{\lambda}_C\ot \id_{A_h})(\gamma^{-1}_{h,h^{-1}}\ot\id_C\ot\id_{A_h})(\overline{\alpha}_{B_h,B_{h^{-1}},C}\ot\id)\overline{c}_{B_h,B_{h^{-1}}}\\ &\overline{\alpha}^{-1}_{B_h\ot_B B_{h^{-1}},B_h,C}(\gamma_{h,h^{-1}}\ot\id_{B_h\ot_B C})(\overline{\lambda}^{-1}_{B_h}\ot \id_C).
  \end{align*}
  Since $\overline{c}_{X,M}$ satisfy 
  \eqref{modfunctor1} and \eqref{modfunctor2}, then isomorphisms $\tau$ satisfy \eqref{typeA-equiv}.
 \epf

Now,  from a $\ca_H$-module we shall construct a type-A datum.

\begin{prop}\label{modcat-to-typeA-datum}  Let $g\in G$ and $H\subseteq G$ be a subgroup. Assume $A\in\ca$ is an algebra, and $(\ca_g)_A$ has an additional structure of  indecomposable exact $\ca_H$-module category such that $\Res^{\ca_H}_\ca (\ca_g)_A$ is indecomposable. There exists $(H^g, \{A_h\}_{h\in H^g}, \beta)$ a type-A datum such that the $\ca_H$-module structure presented in Lemma \ref{typeA-datum-to-modcat0}  coming from this type-A datum is equivalent to the original one.
 \end{prop}
 
 \pf   Lemma \ref{action-equiv} implies that there are  equivalences of $\ca$-modules 
 $$ \Psi_h:\ca_h\boxtimes_\ca (\ca_g)_A \to (\ca_{hg})_A, \quad \otb_h:\ca_h\boxtimes_\ca (\ca_g)_A \to (\ca_g)_A, \text{ for any }h\in H,$$
 where $ \Psi_h$ is the restriction of the tensor product, that is the restriction of the functor $M_{h,g}$ (see Proposition \ref{Res}), and $\otb_h$ is the restriction of the action of $\ca_H$ on $(\ca_g)_A$. This implies that there exists an equivalence of $\ca$-modules 
 $R_{h,g}:(\ca_{hg})_A\to (\ca_g)_A$. Therefore, 
 there exists an invertible $A$-bimodule $A_{g^{-1}h^{-1}g}\in {}_A(\ca_{(h^{-1})^g})_A$ such that   
 $$R_{h,g}(X)=X\ot_A A_{(h^{-1})^g}.$$ The construction of the functor $R_{h,g}$ implies that, there are $\ca$-module natural isomorphisms $\xi_{h,g}: \otb_h\to R_{h,g}\circ \Psi_h$. Note that if $X\in \ca_h$, $M\in\ca_A$, then 
 $$ R_{h,g}\circ \Psi_h(X\boxtimes M)= (X\ot M)\ot_A A_{(h^{-1})^g}.$$
 This is precisely the action $\ott$ defined in Lemma \ref{typeA-datum-to-modcat0}.
Again, Lemma \ref{action-equiv} implies that, for any $f,h\in H$, there are natural module isomorphisms 
 $$m_{f,h}: \otb_{fh} (M_{f,h}\boxtimes \Id_{\C_A})\to \otb_f (\Id_{\C_f}\boxtimes \otb_h). $$
 Whence, using the isomorphisms   $\xi_{h,g}$ we get natural module isomorphisms 
 \begin{equation}\label{iso-funct1} \widetilde{m}_{f,h}:R_{fh,g}  \Psi_{fh} (M_{f,h}\boxtimes \Id_{(\C_g)_A})\to  R_{f,g} \Psi_{f} (\Id_{\ca_f}\boxtimes R_{h,g})(\Id_{\ca_f}\boxtimes  \Psi_{h,g}),
 \end{equation}
 $$ (\widetilde{m}_{f,h})_{X,Z,M} = \big(\id_X\ot  (\xi_{h,g})_{Z,M}\ot \id \big)(\xi_{f,g})_{X,Z\otb M} (m_{f,h})_{X,Z,M}  (\xi^{-1}_{fh,g})_{X\ot Z,M}$$
for any $f,h\in H$, $X\in \C_f, Z\in \C_h$, $M\in\C_A$. Since  isomorphisms $m_{f,h}$ satisfy \eqref{mdgt1}, then isomorphisms  $ \widetilde{m}_{f,h}$ satisfy
\begin{multline}\label{mdgt2} (\widetilde{m}_{f,l})_{X,Y,Z\ott M} (\widetilde{m}_{fl,h})_{X\ot Y,Z,  M} =\\
=(\id_X\ot (\widetilde{m}_{l,h})_{Y,Z,  M}\ot \id) (\widetilde{m}_{f,lh})_{X,Y\ot Z,  M} (\alpha_{X,Y,Z}\ot \id_M\ot \id),
\end{multline}
for any $f,l,h\in H$, $X\in \C_f, Y\in\C_l, Z\in \C_h$. Indeed, denote  $\widetilde{h}=g^{-1}h^{-1}g$. The left hand side of \eqref{mdgt2} is equal to
\begin{align*}
&(\id_X\ot (\xi_{l,g})_{Y,(Z\ot M)\ot_A A_{{\widetilde{h}}}}\ot\id)(\xi_{f,g})_{X,Y\otb ((Z\ot M)\ot_A A_{{\widetilde{h}}})}(m_{f,l})_{X,Y,(Z\ot M)\ot_A A_{{\widetilde{h}}}}\\
&(\xi^{-1}_{fl,g})_{X\ot Y,(Z\otb M)}\circ (\id_{X\ot Y}\ot(\xi_{h,g})_{Z,M}\ot\id)(\xi_{fl,g})_{X\ot Y,Z\otb M}\\
&(m_{fl,h})_{X\ot Y,Z,M}(\xi^{-1}_{flh,g})_{(X\ot Y)\ot Z,M}\\
=&(\id_X\ot (\xi_{l,g})_{Y,(Z\ot M)\ot_A A_{{\widetilde{h}}}}\ot\id)(\xi_{f,g})_{X,Y\otb ((Z\ot M)\ot_A A_{{\widetilde{h}}})}(m_{f,l})_{X,Y,(Z\ot M)\ot_A A_{{\widetilde{h}}}}\\
&(\id_{X\ot Y}\ot(\xi_{h,g})_{Z,M}\ot\id)(m_{fl,h})_{X\ot Y,Z,M}(\xi^{-1}_{flh,g})_{(X\ot Y)\ot Z,M}.\end{align*}
The last equals follows from naturality of $\xi$. Using the naturality of $m$, the above expresssion equals to

\begin{align*}
=&(\id_X\ot (\xi_{l,g})_{Y,(Z\ot M)\ot_A A_{{\widetilde{h}}}}\ot\id)(\xi_{f,g})_{X,Y\otb ((Z\ot M)\ot_A A_{{\widetilde{h}}})}(\id_{X}\ot\id_{ Y}\ot(\xi_{h,g})_{Z,M})\\
&(m_{f,l})_{X,Y,Z\otb M}(m_{fl,h})_{X\ot Y,Z,M}(\xi^{-1}_{flh,g})_{(X\ot Y)\ot Z,M}\\
=&(\id_X\ot (\xi_{l,g})_{Y,(Z\ot M)\ot_A A_{{\widetilde{h}}}}\ot\id)(\xi_{f,g})_{X,Y\otb ((Z\ot M)\ot_A A_{{\widetilde{h}}})}(\id_{X}\ot\id_{ Y}\ot(\xi_{h,g})_{Z,M})\\
&(\id_X\ot (m_{l,h})_{Y,Z,M})(m_{f,lh})_{X,Y\ot Z, M}(\alpha_{X,Y,Z}\ot\id_M)(\xi^{-1}_{flh,g})_{(X\ot Y)\ot Z,M}\\
=&(\id_X\ot (\xi_{l,g})_{Y,(Z\ot M)\ot_A A_{{\widetilde{h}}}}\ot\id)(\xi_{f,g})_{X,Y\otb ((Z\ot M)\ot_A A_{{\widetilde{h}}})}(\id_{X}\ot\id_{ Y}\ot(\xi_{h,g})_{Z,M})\\
&(\id_X\ot (m_{l,h})_{Y,Z,M})(m_{f,lh})_{X,Y\ot Z, M}(\xi^{-1}_{flh,g})_{(X\ot Y)\ot Z,M}(\alpha_{X\ot Y,Z,M}\ot\id)\end{align*}\begin{align*}
=&(\id_X\ot (\xi_{l,g})_{Y,(Z\ot M)\ot_A A_{{\widetilde{h}}}}\ot\id)(\xi_{f,g})_{X,Y\otb ((Z\ot M)\ot_A A_{{\widetilde{h}}})}(\id_{X}\ot\id_{ Y}\ot(\xi_{h,g})_{Z,M})\\
&(\id_X\ot (m_{l,h})_{Y,Z,M})(\xi_{f,g}^{-1})_{X,(Y\ot Z)\otb M}(\id_X\ot(\xi^{-1}_{lh,g})_{(Y\ot Z,M}\ot\id)\\
&(\widetilde{m}_{f,lh})_{X,Y\ot Z, M}(\alpha_{X\ot Y,Z,M}\ot\id)\\
=&(\id_X\ot (\xi_{l,g})_{Y,(Z\ot M)\ot_A A_{{\widetilde{h}}}}\ot\id)(\xi_{f,g})_{X,Y\otb ((Z\ot M)\ot_A A_{{\widetilde{h}}})}(\id_{X}\ot\id_{ Y}\ot(\xi_{h,g})_{Z,M})\\
&(\xi_{f,g}^{-1})_{X,(Y\otb Z)\otb M}(\id_X\ot (m_{l,h})_{Y,Z,M}\ot\id)(\id_X\ot(\xi^{-1}_{lh,g})_{Y\ot Z,M}\ot\id)\\
&(\widetilde{m}_{f,lh})_{X,Y\ot Z, M}(\alpha_{X\ot Y,Z,M}\ot\id)\\
=&(\id_X\ot (\xi_{l,g})_{Y,(Z\ot M)\ot_A A_{{\widetilde{h}}}}\ot\id)(\xi_{f,g})_{X,Y\otb ((Z\ot M)\ot_A A_{{\widetilde{h}}})}(\id_{X}\ot\id_{ Y}\ot(\xi_{h,g})_{Z,M})\\
&(\xi_{f,g}^{-1})_{X,(Y\otb Z)\otb M}(\id_X\ot(\xi_{l,g}^{-1})_{Y,Z\otb M}\ot\id)(\id_X\ot (\xi^{-1}_{h,g})_{Y,Z\otb M}\ot\id)\\
&(\id_X\ot(\widetilde{m}_{l,h})_{Y,Z,M}\ot\id)(\widetilde{m}_{f,lh})_{X,Y\ot Z, M}(\alpha_{X\ot Y,Z,M}\ot\id)\\
=&(\id_X\ot (\xi_{l,g})_{Y,(Z\ot M)\ot_A A_{{\widetilde{h}}}}\ot\id)(\id_X\ot (\xi_{h,g}^{-1})_{Y,Z\otb M}\ot\id)(\id_X\ot(\xi_{l,g}^{-1})_{Y,Z\otb M}\ot\id)\\
&(\id_X\ot(\xi^{-1}_{h,g})_{Y,Z\otb M}\ot\id))(\id_X\ot(\widetilde{m}_{l,h})_{Y,Z,M}\ot\id)(\widetilde{m}_{f,lh})_{X,Y\ot Z, M}(\alpha_{X\ot Y,Z,M}\ot\id)\\
=&(\id_X\ot(\widetilde{m}_{l,h})_{Y,Z,M}\ot\id)(\widetilde{m}_{f,lh})_{X,Y\ot Z, M}(\alpha_{X\ot Y,Z,M}\ot\id),
\end{align*}
the second equality follows from \eqref{mdgt1}; the third, fifth, seventh and eight from naturality of $\xi$, the fourth and sixth from 
definition of $\widetilde{m}$.
\eq
\smallbreak

The associativity of the tensor category $\Do$ imply that there are  $\ca$-module natural isomoprhisms 
 $$  \Psi_{f} (\Id_{\ca_f}\boxtimes R_{h,g}) \simeq R_{fhf^{-1},fg} M_{f,hg}, \ \ \Psi_{fh}(M_{f,h}\btc \Id_{(\C_g)_A})\simeq M_{f,hg}(\id_{\C_f}\btc \Psi_h).$$
Thus, it follows from the equivalence \eqref{iso-funct1} that there are 
$\ca$-module natural isomorphisms $ \gamma_{f,h}:R_{fh,g}\simeq  R_{f,g} \circ R_{fhf^{-1},fg}$. Using  Proposition  \ref{natu-bimod}, we obtain that there are $A$-bimodule isomorphisms $\beta_{(h^{-1})^g, (f^{-1})^g}: A_{(h^{-1}f^{-1})^g}\to A_{(h^{-1})^g}\ot_A A_{(f^{-1})^g} $ such that
$$(\gamma_{f,h})_X = \walph^{-1}_{X,A_{(h^{-1})^g},A_{(f^{-1})^g}} (\id_X\ot \beta_{(h^{-1})^g, (f^{-1})^g}).$$
Reconstructing the isomorphisms $\widetilde{m}_{f,h}$, we obtain that 
\begin{align*}(\widetilde{m}_{f,h})_{X,Y,M}=&(\balph_{X,Y\ot M, A_{(h^{-1})^g}}\ot\id)\walph^{-1}_{X\ot (Y\ot M), A_{(h^{-1})^g}, A_{(f^{-1})^g}}\\&(\id\ot \beta_{(h^{-1})^g,(f^{-1})^g})(\alpha_{X,Y,M }\ot \id ).
\end{align*}
Note that this formula coincides with the associativity defined in \eqref{associ-typeA-b}.
 Equation \eqref{mdgt2} implies that morphisms $\beta$ satisfy \eqref{beta-assoc}. Hence, we get a type-A datum $(H^g, \{A_h\}_{h\in H^g}, \beta)$ 
 such that, by construction, the $\ca_H$-module structure presented in Lemma \ref{typeA-datum-to-modcat0}  coming from this type-A datum is equivalent to the original one.
\epf

\begin{defi}\label{defi:associated-typeA} Let $H\subseteq G$ be a subgroup and let $\No$ be an indecomposable exact $\ca_H$-module such that $Res^{\C_H}_\C\N$ remains indecomposable as a $\ca$-module. There exists an algebra $A\in\ca$ such that $\Res^{\C_H}_\C\No\simeq\ca_A$ as $\ca$-modules. The type-A data  $(H, \{A_h\}_{h\in H}, \beta) $ obtained in Proposition \ref{modcat-to-typeA-datum}, using $g=1$, is the \emph{associated} type-A data to $\No$.
\end{defi}

 It follows from Lemma \ref{typeA-datum-to-modcat}, that the equivalence class of the type-A data associated to a  $\ca_H$-module $\N$, does not depend on the Morita class of the algebra $A$ such that $\Res^{\C_H}_\C\No\simeq\ca_A$ as $\ca$-modules. 
 
\medbreak

Now, we shall explain  the classification of indecomposable exact $\D$-modules obtained in \cite{Ga2} and \cite{MM}. This classification will be done in two steps. In the first step we will associate to any 
  indecomposable exact $\D$-module a pair $(H,\N)$, where $H\subseteq G$ is a subgroup, and $\N$ is an indecomposable exact $\C_H$-module such that the restriction $Res^{\C_H}_\C\N$ remains indecomposable as a $\C$-module.
In the second step, using Proprosition \ref{modcat-to-typeA-datum}, we shall associate, to any such pair $(H,\N)$ a type-A datum for $Res^{\C_H}_\C\N$.

\subsection{First step} Let $H$ be a subgroup of $G$, and let $\N$ be an indecomposable exact $\C_H$-module such that $\M=Res^{\C_H}_\C\N $ is an indecomposable %exact 
$\C$-module. Under these assumptions, we shall call $(H,\N)$ a \textit{type-1 pair.} 

Note that if $(H,\N)$ is a type-1 pair, then, for any $g\in G$, the category $\C_{gH}\boxtimes_{\C_H} \N$ has an action of $\C_{gHg^{-1}}$, such that 
$$ Res^{\C_{gHg^{-1}}}_\C \big(\C_{gH}\boxtimes_{\C_H} \N\big)\simeq \C_g\boxtimes_\C \M.  $$
By Lemma \ref{action-equiv0}, $\C_g\boxtimes_\C \M$ is an indecomposable $\C$-module, then $(gHg^{-1}, \C_{gH} \boxtimes_{\C_H}\N)$ is again a type-1 pair.

\begin{defi} Two type-1 pairs $(H,\N)$, $(F,\N')$ are \textit{equivalent} if there exists $g\in G$ such that 
\begin{itemize}
\item[$\bullet$]  $H=gFg^{-1}$, and
\item[$\bullet$]  there is an equivalence $\N\simeq\C_{gF} \boxtimes_{\C_F} \N'$ of $\C_H$-modules.
\end{itemize} 

\end{defi}

It follows from Lemma \ref{ind-rest} (2), that if  $(H,\N)$ is a type-1 pair,  then $\Ind^\D_{\C_H} \N$ is an indecomposable exact $\D$-module. We shall see that this establishes a bijective correspondence between equivalence classes of indecomposable exact $\D$-modules and equivalence classes of type-1 pairs.

\medbreak

Let's start with  an exact $\D$-module  $\N$. It follows from Lemma \ref{ind-rest} (1) that $Res^\D_\C\N$ is exact. Then, we can decompose it as $Res^\D_\C\N=\oplus_{i=1}^n \N_i$, into a direct sum of indecomposable exact $\C$-modules. Denoted by $\overline{\N_i}$ the equivalence class of the $\C$-module $\N_i$, and by $X_\N=\{ \overline{\N_i}:i=1\dots n\}$.

The group $G$ acts on the set $X_\N$. Namely, $g\cdot \overline{\N_i}=\overline{\N_j}$, if 
$\C_g\bt_{\C}\N_i\simeq \N_j$ as $\C$-modules. This action is transitive since $\N$ is indecomposable as a $\D$-module. The next result is well-known.

\begin{lema}\label{X-set} If $\N, \M$ are equivalent $\D$-modules, then $X_\M\simeq X_\N$ as $G$-sets. If $H=\{f\in G| f\cdot \overline{\N_1}=\overline{\N_1}\}$, then $X_\N\simeq G/H$ as $G$-sets. \qed
\end{lema}

\begin{prop}\label{equiv-pairs}\cite[Proposition 4.6]{Ga2} Let $(H, \N)$ and $(F,\M)$ be two type-1 pairs. The following statements are equivalent.
\begin{itemize}
\item[1.] There exists an equivalence of $\D$-modules
$\Ind^{\D}_{\C_H} \N \simeq \Ind^{\D}_{\C_F} \M.$

\item[2.] The type-1 pairs $(H, \N)$, $(F,\M)$ are equivalent.
\end{itemize}

\end{prop}
\begin{proof} $(1)\Rightarrow (2)$
If $\Ind^{\D}_{\C_H} \N, \Ind^{\D}_{\C_F} \M$ are equivalent as $\C$-modules, then, by Lemma \ref{X-set},  $ X_{\Ind^\D_{\C_{H}}\N}\simeq X_{\Ind^\D_{\C_{F}}\M}$ as $G$-sets.

Decompose $Res^\D_\C(\Ind^{\D}_{\C_H} \N)$ into a direct sum $\oplus_{i=1}^n \A_i$ of indecomposable $\C$-modules. Since the module  $Res^{\C_H}_\C\N$ is included in $Res^\D_\C(\Ind^{\D}_{\C_H} \N)$, and it remains indecomposable as $\C$-module, we can assume $\A_1\simeq Res^{\C_H}_\C\N$. It is not difficult to see that $\stab(\overline{Res^{\C_H}_\C\N})
=H$.  By Lemma \ref{X-set}, $G/H\simeq G/F$, thus, there exist $g\in G$ such that $H=gFg^{-1}$.

The restriction of the given equivalence, gives us $$\C_{gF}\bt_{\C_{F}}\M\simeq \C_H\bt_{\C_H}\N\simeq\N,$$ 
as $\C_H$-modules, where the $\C_H$-action over $\C_{gF}$ is induced by the tensor product of $\C$, and is well defined since $Hg=gF$.

\medbreak

$(2)\Rightarrow (1)$ Let $g\in G$ such that $H=gFg^{-1}$ and $\C_{gF}\bt_{\C_{F}}\M\simeq \N$ as $\C_H$-modules. For any $a\in G$, there are equivalences of right $\C_{F}$-modules
\begin{align}\label{equiv-as-mt1}
\C_{aH}\bt_{\C_H}\C_{gF}&\simeq (\C_{a}\bt_{\C} \C_H )\bt_{\C_H}\C_{gF}\simeq \C_{a}\btc(\C_{H}\bt_{\C_H}\C_{gF})
\simeq \C_{a}\btc\C_{gF}\simeq \C_{agF},
\end{align}
where the right $\C_F$-module structrure is given by tensor product.

Let $\{t_1,...,t_n\}$ be a set of representative of the cosets of $G/F$, thus $\{t_1g^{-1},...,t_ng^{-1}\}$ is a set of representative of the cosets of $G/H$.  Using \eqref{equiv-as-mt1}, we have $\D$-module equivalences
\begin{align*}
\D\bt_{\C_{F}}\M&\simeq \oplus_{i=1}^n\C_{t_iF}\bt_{\C_{F}}\M\simeq \oplus_{i=1}^n(\C_{t_ig^{-1}H}\bt_{\C_{H}}\C_{gF})\bt_{\C_{F}}\M\\
&\simeq \oplus_{i=1}^n\C_{t_ig^{-1}H}\bt_{\C_{H}}(\C_{gF}\bt_{\C_{F}}\M)\\&\simeq \oplus_{i=1}^n\C_{t_ig^{-1}H}\bt_{\C_{H}}\N\simeq\D\bt_{\C_{H}}\N.
\end{align*}
\end{proof}

Now, we shall prove that the map
$$(H,\N)  \longmapsto \Ind^{\D}_{\C_H} \N$$
gives a first step to classify indecomposable exact $\D$-modules.
 The proof of the following Theorem follows the same steps as the proof of \cite[Proposition 12]{MM} in the semisimple case.

\begin{teo}\label{classi-pairs}\cite[Proposition 12]{MM} There exists a bijection between 

\begin{itemize}
    \item equivalence classes of indecomposable exact $\D$-modules, and
    
    \item  equivalence classes of type-1 pairs $(H,\N)$.
\end{itemize} \end{teo}

\begin{proof} Take a type-1 pair $(H,\N)$. Then $\Ind^{\D}_{\C_H} \N$ is an indecomposable exact $\D$-module. By Proposition \ref{equiv-pairs}, the equivalence class of this module category does not depend on the equivalence class of the pair $(H,\N)$. %The type-1 pair associated to  $\Ind_{\C_H}^\D\N$ is $(H,\N)$. This follows from the first part of the proof of Proposition \ref{equiv-pairs}. 

\medbreak

Let $\M$ be an indecomposable exact $\D$-module. We shall construct a type-1 pair $(H,\N)$ such that $\M\simeq \Ind^{\D}_{\C_H} \N$. Let $Res^\D_\C\M=\oplus_{i=1}^n \M_i$
be a decomposition of indecomposable exact $\C$-modules. Consider the action of $G$ over $X_\M$ as described before.

 Let $H:=\text{Stab}(\M_1)=\{f\in G: \C_f\btc\M_1\simeq\M_1\text{ as $\C$-modules}\}$. Set $\N:=\M_1$. The action $\otb:\D\times \M\to \M$ restricted to $\otb:\C_H\times \N\to \N$, induced a structure of $\C_H$-module over $\N$. $\N$ is an exact $\C_H$-module since $\M$ is an exact $\C$-module. $\N$  is an  indecomposable $\C_H$-module since $\N$ is indecomposable as a $\C$-module. Hence, we obtain a type-1 pair $(H,\N)$. It follows from Lemma \ref{acti-equivalence}  that $\M\simeq\Ind_{\C_H}^\D\N$ as $\D$-modules.
\end{proof}

\subsection{Second step} To any type-1 pair $(H,\N)$, there is associated a type-A datum, namely the type-A data associated to the $\ca_H$-module $\N$. See Definition \ref{defi:associated-typeA}. We shall introduce a new equivalence relation of type-A data, such that there exists a bijection between equivalence classes of type-1 pairs and equivalence classes of type-A data.

For this, we shall first need to explicitly give the  type-A data ass\-ociated to 
$(gHg^{-1}, \C_{gH} \boxtimes_{\C_h}\N)$ in terms of the type-A data associated to $(H,\N)$.

\begin{rmk}\label{rmk-g-act} Let $A\in\ca$ be an algebra such that $\ca_A$ is an indecomposable exact $\ca$-module. Let $(H, \{A_h\}_{h\in H}, \beta)$ be a type-A data for $\ca_A$, then using  Lemma \ref{typeA-datum-to-modcat0} we obtain that
\begin{itemize}
    \item $\ca_A$ has a structure of $\ca_H$-module, taking $g=1$, and
    \item $(\ca_g)_A$ has a structure of $\ca_{gHg^{-1}}$-module, using $gHg^{-1}$ instead of $H$. This action, according to \eqref{mod-action-cA}, is
$$Z\widetilde{\ot} M= (Z\ot M)\ot_A A_{h^{-1}}, \ \ \text{for any }Z\in \C_{ghg^{-1} }, M\in (\ca_g)_A.$$
\end{itemize}
\end{rmk}
\begin{prop}\label{typeA-for-changed} Let $(H,\N)$ be a type-1 pair, and $A\in \ca$ be an algebra such that $\Res^{\C_H}_\C\N\simeq \ca_A$. Let $(H, \{A_h\}_{h\in H}, \beta)$ be an algebra datum for $\C_A$.
There exists an equivalence   of $\ca_{gHg^{-1}}$-modules  $$\C_{gH} \boxtimes_{\C_H}\N\simeq (\ca_g)_A.$$
Here the structure of $\ca_{gHg^{-1}}$-module on $(\ca_g)_A$ is the one presented in Remark \ref{rmk-g-act}.
\end{prop}
\begin{proof} To avoid complicated calculations, we shall assume that the category $\Do$ is strict. We shall freely use the identification $ \C_{gH}\simeq \C_g\boxtimes_{\C} \C_H$, see Proposition \ref{Res} (2).
There exists a functor $\Phi: \C_{gH} \boxtimes_{\C_H} \ca_A \to (\ca_g)_A$ such that
$$\Phi((X\boxtimes Y) \boxtimes  M)= X\ot Y\ot M\ot_A A_{h^{-1}},$$
for any $X\in \C_g$, $Y\in\C_h$, $h\in H$, $M\in \C_A$. The functor $\Phi$ is an equivalence of categories, since it is the composition 
$$\C_{gH} \boxtimes_{\C_H} \ca_A \xrightarrow{\overline{M}_{g,H}\boxtimes \Id} 
\C_g \boxtimes_{\C} \C_H \boxtimes_{\C_H} \ca_A \xrightarrow{\Id  \boxtimes \otb}   \C_g \boxtimes_{\C}\ca_A  \xrightarrow{\ot}  (\ca_g)_A,$$
where $\overline{M}_{g,H}: \C_{gH} \to \C_g \boxtimes_{\C} \C_H$ is a quasi-inverse of the equivalence presented in Proposition \ref{Res} (2).
The fact that the action functor $\otb:\C_H \boxtimes_{\C_H} \ca_A \to \ca_A  $  and the restriction of the  tensor product $\C_g \boxtimes_{\C} \ca_A \to (\ca_g)_A$ are equivalences follows from  Lemma \ref{action-equiv0} using that $\C_A$ is indecomposable. The functor $\Phi$ has a module structure as follows. If $f,h\in H$,  $Z\in \C_{gfg^{-1}}$, $U\in  \C_{gh}$, $M\in \C_A$, then
$$\Phi(Z\otb (U\boxtimes M))= Z\ot U\ot M\ot_A A_{(fh)^{-1}}, $$
$$Z\widetilde{\ot}\Phi(U\boxtimes M) = Z\widetilde{\ot} (U\ot M\ot_A A_{h^{-1}})= Z\ot U\ot M\ot_A A_{h^{-1}}\ot_A A_{f^{-1}}.$$
Define the natural isomorphisms $$c_{Z, U\boxtimes M}: Z\ot U\ot M\ot_A A_{(fh)^{-1}}\to Z\ot U\ot M\ot_A A_{h^{-1}}\ot_A A_{f^{-1}}$$ as $c_{Z, U\boxtimes M}=\id\ot \beta_{h^{-1},f^{-1}}$.
In this case, equation \eqref{beta-assoc} is equivalent to \eqref{modfunctor1} and \eqref{modfunctor2}.  Hence $(\Phi, c)$ is a module equivalence.
\end{proof}

\begin{defi}\label{defi:G-equiv} Two type-A data $(H, \{A_h\}_{h\in H}, \beta)$, $(F, \{B_f\}_{f\in F}, \gamma)$ are $G$-\textit{equivalent} if there exists $g\in G$, and  an invertible $(B,A)$-bimodule $C\in {}_B(\C_g)_{A}$ together with $(B,A)$-bimodule isomorphisms $\tau_h: B_{ghg^{-1}}\ot_B C \to C \ot_A A_h$, $h\in H$, such that $F=gHg^{-1}$, and for any $h,l\in H$
\begin{equation}\label{Gequiv1} 
\overline{\rho}_C\tau_1= \overline{\lambda}_C,
\end{equation}
\begin{equation}\label{Gequiv2}\begin{split}
    (\id\ot \beta_{h,l})\tau_{hl}=\walph_{C,A_h,A_l}(\tau_h\ot\id) &\walph^{-1}_{B_{ghg^{-1}},C,A_l} (\id\ot \tau_l)\walph_{B_{ghg^{-1}},B_{glg^{-1}},C}\\&(\gamma_{ghg^{-1},glg^{-1}}\ot\id).
\end{split} 
\end{equation}

\end{defi}

\begin{rmk} Two equivalent type-A data, according to Definition \ref{def:typea}, are G-equivalent. One has to take $g=1$.
\end{rmk}

Let $(F, \{B_f\}_{f\in F}, \gamma)$  be a type-A datum for $\ca_B$. By Lemma \ref{typeA-datum-to-modcat0}, $\C_B$ has structure of $\C_F$-module. Let $(H, \{A_h\}_{h\in H}, \beta)$ be a type-A datum for $\C_A$ then, using remark \ref{rmk-g-act}, it follows that for any $g\in G$, $(\C_g)_A$ has a structure of $\C_{gHg^{-1}}$-module. These two actions will be used in the next Proposition.

\begin{prop}\label{g-eq-tA} Let $A, B\in\C$ be algebras such that $\C_A, \C_B$ are indecomposable exact $\C$-modules. Let $(H, \{A_h\}_{h\in H}, \beta)$ be a type-A datum for $\C_A$ and $(F, \{B_f\}_{f\in F}, \gamma)$ be  type-A data for $\C_B$. They are $G$-equivalent, if and only if, there exists $g\in G$ such that $(\C_g)_A$, $\C_B$ are equivalent as $\C_{gHg^{-1}}$-modules.
\end{prop}
\pf The proof follows {\it mutatis mutandis} the proof of Lemma \ref{typeA-datum-to-modcat}.\epf

Combining Theorem \ref{classi-pairs} and Proposition \ref{g-eq-tA} we get the main result of the paper.
\begin{teo}\label{classi-pairsAand1} Let $G$ be a finite group and $\D=\oplus_{g\in G} \C_g$ be a $G$-graded extension of $\C=\ca_1$. There exists a bijection between 

\begin{itemize}
   \item equivalence classes of indecomposable exact $\D$-modules, 
    \item equivalence classes of type-1 pairs, and
    
    \item  $G$-equivalence classes of type-A data. 
\end{itemize} \end{teo}
 \pf In presence of Theorem \ref{classi-pairs}, we only have to present a bijection between equivalence classes of type-1 pairs and $G$-equivalence classes of type-A data.  Let $(H,\N)$ be a type-1 pair, and $A\in \ca$ be an algebra such that $\Res^{\C_H}_\C\N\simeq \ca_A$. We have associated to $(H,\N)$ a type-A datum $(H, \{A_h\}_{h\in H}, \beta) $, see Definition \ref{defi:associated-typeA}. 
 
 This map does not depend on the  equivalence class of the chosen type-1 pair. Indeed, let $(F,\No')$ be a type-1 pair equivalent to $(H,\N)$. Then, there exists $g\in G$ such that $F=gHg^{-1}$ and there exists an equivalence $\No'\simeq \C_{gH}\boxtimes_{\C_H} \No$ of $\C_{gHg^{-1}}$-modules. It follows from Proposition \ref{typeA-for-changed} that $\C_{gH}\boxtimes_{\C_H} \No\simeq (\C_g)_A$ as $\C_{gHg^{-1}}$-modules. Let $(gHg^{-1}, \{B_f\}_{f\in gHg^{-1}}, \gamma)$ be the type-A datum associated to $\N'$. Since $\N'\simeq \C_B$ as $\C_{gHg^{-1}}$-modules, then $\C_B\simeq (\C_g)_A$ as $\C_{gHg^{-1}}$-modules. It follows from Proposition \ref{g-eq-tA} that type-A data $(gHg^{-1}, \{B_f\}_{f\in gHg^{-1}}, \gamma)$, $(H, \{A_h\}_{h\in H}, \beta) $ are $G$-equivalent. 
 \epf

The explicit description of $\D$-modules from the type-A datum allows us to obtain the next result. Recall that equivalence classes of fiber functors for a tensor category $\D$ are in correspondence with equivalence classes of semisimple $\D$-module categories of rank 1.
\begin{cor}\label{fiber-funct} Let $G$ be a finite group and $\D=\oplus_{g\in G} \C_g$ be a $G$-graded extension of $\C=\ca_1$. There exists a bijection between
\begin{itemize}
   \item Equivalence classes of fiber functors $F:\D\to \vect_\ku$, and
   
  \item $G$-equivalence classes of type-A datum $(G, \{A_h\}_{h\in G}, \beta)$ such that $\C_A$ is a semisimple category of rank 1.\qed
\end{itemize}
\end{cor}

\subsection{Pointed fusion categories}\label{pointed-exa}  In this Section we shall show that our results, when applied to a pointed fusion category, agree with the results obtained in \cite{Na1}.  Pointed fusion categories are parameterized by  pairs $(G,\omega)$, where $G$ is a finite group, and $\omega\in H^3(G,\ku^{\times})$ is a 3-cocycle. The tensor category $\D=\C(G,\omega)$,  described in  Example \ref{pointed-fusion},  is a $G$-extension of the category $\vect_\ku$.

\medbreak

We first described the equivalence classes of type-A data for these categories.

\begin{lema} There is a bijection between
\begin{itemize}
   \item equivalence classes of type-A data, and
     \item pairs $(H,\beta)$, where $H\subseteq G$ is a group, and $\beta\in H^2(H,\ku^{\times})$ is a homology class such that $d\beta\, \omega\mid_{H\times H\times H}=1$.
\end{itemize}  
\end{lema}
\pf We will only sketch the proof. Let $H\subseteq G$ be a subgroup, and $\beta\in H^2(H,\ku^{\times})$  such that $d\beta\, \omega\mid_{H\times H\times H}=1$. Then, $(H, \{\ku_h\}_{h\in H}, \beta')$ is a type-A data, where $\ku_h$ is the field $\ku$, concentrated in degree $h$, and $\beta'_{f,h}:\ku_{fh}\to \ku_f\ot \ku_h$ is defined by $\beta'_{f,h}(1)= \beta_{f,h} 1\ot 1.$

Let  $(H, \{A_h\}_{h\in H}, \beta)$ be a type-A data, this means that  $H\subseteq G$ is a subgroup and $A_1=A\in \vect_\ku$ is an algebra such that $(\vect_\ku)_A$ is an indecomposable exact $\vect_\ku$-module. This means that $A=\ku$. Hence,  for any $h\in H$, $A_h$ is a 1-dimensional vector space. Let  $\tau_h:A_h\to \ku_h$ be some isomorphism, for any $h\in H$. If we define $\widetilde{\beta}_{f,h}:\ku_{fh}\to \ku_f\ot \ku_h$ as 
$$\widetilde{\beta}_{f,h}=(\tau_f\ot\tau_h)\beta_{f,h} \tau^{-1}_{fh}, $$
then $(H, \{A_h\}_{h\in H}, \beta)$ is equivalent to $(H, \{\ku_h\}_{h\in H}, \widetilde{\beta})$. Equation \eqref{beta-assoc} implies that $d\widetilde{\beta}\, \omega\mid_{H\times H\times H}=1$.
\epf

Theorem \ref{classi-pairsAand1} implies, in this case, that indecomposable exact $\D$-modules are parametrized by pairs $(H,\beta)$, where $H\subseteq G$ is a group, and $\beta\in C^2(H,\ku^{\times})$ is a 2-cochain such that $d\beta\, \omega=1$. This parametrization coincides with the one given in \cite{EO}, \cite{Na1}.

For such pair $(H,\beta)$, denote $\M_0(H,\beta)$ the associated indecomposable exact $\D$-module. As abelian categories $\M_0(H,\beta)=\C(G,\omega)_{\ku_{\beta}H},$ the category of right $\ku_{\beta}H$-modules. 

Let $(H, \{\ku_h\}_{h\in H}, \beta)$ be a type-A datum. The algebra $\widehat{A}$ described in Proposition \ref{typeA-algebra-struc}, in this case, coincides with the twisted group algebra $\ku_{\beta} H$. This implies, using Proposition \ref{alg-mod},  that the module category corresponding to $(H, \{\ku_h\}_{h\in H}, \beta)$ in Theorem \ref{classi-pairsAand1} is precisely $\M_0(H,\beta)$.

Let $(L,\xi)$ be another pair, where $L\subseteq G$ is a subgroup, and $\xi\in H^2(L,\ku^{\times})$ is such that $d\xi\omega=1.$ Theorem \ref{classi-pairsAand1} implies that there exists an equivalence $\M_0(H,\beta)\simeq  \M_0(L,\xi)$ of $\C(G,\omega)$-modules, if and only if, the type-A data $(H, \{\ku_h\}_{h\in H}, \beta)$, $(L, \{\ku_l\}_{l\in L}, \xi)$ are $G$-equivalent. According to Definition \ref{defi:G-equiv}, this means that, there exists $g\in G$ such that $L=gHg^{-1}$, and for any $h\in H$ there are scalars $\tau_h\in \ku^{\times} $  such that
 $\tau_1=1,$ and
 \begin{equation}\label{beta-xi}\beta_{h,l} \tau_{hl}=\omega(g,h,l)\tau_h \omega^{-1}(ghg^{-1},g,l)\tau_l\omega(ghg^{-1},glg^{-1},g)\xi_{ghg^{-1},glg^{-1}},\end{equation}
for any $h,l\in H$. Following \cite{Na1}, we define the 2-cochain $\Omega_g:H\times H\to \ku$ as
$$\Omega_g(h,l)=\frac{\omega(g,h,l)\omega(ghg^{-1},glg^{-1},g) }{\omega(ghg^{-1},g,l)}. $$
Equation \eqref{beta-xi} implies $\beta=\Omega_g \xi^g $ in $H^2(H,\ku^{\times})$.  Hence, we recover the next result.

\begin{teo}\cite[Thm 1.1]{Na1}
Assume  $L, H\subseteq G$ are two groups, and $\beta\in C^2(H,\ku^{\times})$, 
$\xi\in C^2(L,\ku^{\times})$ are 2-cochains such that $d\beta =\omega^{-1}= d\xi $. There exists an equivalence of module categories between $\M_0(H,\beta), \M_0(L,\xi)$, if and only if, there exists $g\in G$ such that $H=gLg^{-1},$ and the class of $\beta^{-1}\xi^g \Omega_g\mid_{L\times L}$ is trivial in $H^2(H,\ku^{\times})$.\qed
\end{teo}


\begin{thebibliography}{AEGPk}




\bibitem{DN} {\sc A. Davydov} and {\sc. D. Nikshych}, \emph{The Picard crossed module of a braided tensor category}, Algebra and Number Theory \textbf{7} no. 6 (2013), 1365--1403.



\bibitem{EG} {\sc P. Etingof} and {\sc S. Gelaki}, \emph{Exact sequences of tensor categories with respect to a module category}, Adv. math. \textbf{308} (2017),  1187--1208. 

\bibitem{EGNO}  {\sc P. Etingof}, {\sc S. Gelaki}, {\sc D. Nikshych} and {\sc V. Ostrik}, \emph{Tensor categories}, Lectures notes (2009) 80--83. http://www-math.mit.edu/$\sim$etingof/tenscat1.pdf


\bibitem{ENO}  {\sc P. Etingof}, {\sc D. Nikshych} and {\sc V. Ostrik}, \emph{Fusion categories and homotopy theory}, Quantum Topol. \textbf{1} no. 3 (2010), 209--273.

\bibitem{EO} {\sc P. Etingof} and {\sc V. Ostrik},
\emph{Finite tensor categories}, Mosc. Math. J. \textbf{4} no. 3 (2004)
, 627--654.




\bibitem{Ga} {\sc C. Galindo}, \emph{Clifford theory for tensor categories},
J. London Math. Soc. \textbf{83} no. 1 (2011), 57--78.






\bibitem{Ga2} {\sc C. Galindo}, \emph{Clifford theory for graded fusion categories}, Israel J. Math. \textbf{192} no. 2 (2012), 841--867.


\bibitem{DSS} {\sc C.L. Douglas},  {\sc C. Schommer-Pries} and {\sc N. Snyder}, \emph{The balanced tensor product of module categories}, preprint arXiv:1406.4204.


\bibitem{Gr}  \textsc{J. Greenough}. \emph{Monoidal 2-structure of Bimodules}, J. Algebra \textbf{324} (2010), 1818--1859.
 



\bibitem{MM} {\sc E. Meir} and {\sc E. Musicantov},
 \emph{modules over graded fusion categories},  J. Pure  Appl. Algebra \textbf{216} (2012), 2449--2466.




\bibitem{Na1} {\sc S. Natale}, \emph{On the equivalence of module categories over a group-theoretical fusion category}, SIGMA, Symmetry Integrability Geom. Methods Appl. \textbf{13} Paper 042, 9 p. (2017).




\bibitem{TY} {\sc D. Tambara, S. Yamagami,} \textit{Tensor categories with fusion rules of self-duality for finite abelian groups}, J. Algebra \textbf{209} no. 2 (1998), 692--707.


\end{thebibliography}
\end{document}